\documentclass[10pt,reqno]{amsart}
\usepackage{amssymb,mathrsfs,color}

\usepackage{amsmath, amsfonts, amsthm, verbatim}
\usepackage{epstopdf, mathtools}
\mathtoolsset{showonlyrefs}

\usepackage{enumerate}

\usepackage{cite}

\renewcommand{\Re}{\mathrm{Re}}

\newcommand{\ud}{\mathrm{d}}

\newcommand{\eps}{\epsilon}

\newcommand{\vp}{\varphi}

\newcommand{\bfg}{{\bf g}}

\newcommand{\calD}{\mathcal D}

\newcommand{\calK}{\mathcal K}

\definecolor{deepgreen}{cmyk}{0.1,0.7,0.8,0.1}

\newcommand{\E}{\mathcal{E}}

\newcommand{\HH}{\mathcal{H}}

\newcommand{\LL}{\mathcal{L}}

\newcommand{\CC}{\mathscr{C}}

\newcommand{\N}{\mathbb{N}}
\newcommand{\R}{\mathbb{R}}
\newcommand{\Sp}{\mathbb{S}}
\newcommand{\Z}{\mathbb{Z}}

\newcommand{\g}{\mathbf{g}}

\newcommand{\etab}{\boldsymbol{\eta}}

\newcommand{\al}{\alpha}
\newcommand{\be}{\beta}
\newcommand{\ga}{\gamma}
\newcommand{\de}{\delta}
\newcommand{\e}{\varepsilon}
\newcommand{\fy}{\varphi}
\newcommand{\om}{\omega}
\newcommand{\la}{\lambda}

\newcommand{\s}{\sigma}

\newcommand{\vphi}{\varphi}

\newcommand{\De}{\Delta}
\newcommand{\Om}{\Omega}

\newcommand{\La}{\Lambda}

\newcommand{\p}{\partial}
\newcommand{\na}{\nabla}

\newcommand{\rar}{\rightarrow}
\makeatletter

\newcommand{\Rmnum}[1]{\expandafter\@slowromancap\romannumeral #1@}
\makeatother

\newcommand{\I}{\infty}

\newcommand{\ti}{\widetilde}

\newcommand{\LR}[1]{{\langle #1 \rangle}}
\newcommand{\ang}[1]{\left\langle{#1}\right\rangle}
\newcommand{\abs}[1]{\left\lvert{#1}\right\rvert}

\newcommand{\ant}[1]{\begin{align*}\begin{split} #1 \end{split}\end{align*}}
\newcommand{\EQ}[1]{\begin{equation}\begin{split} #1 \end{split}\end{equation}}

\newcommand{\Del}[1]{}

\numberwithin{equation}{section}

\newtheorem{thm}{Theorem}[section]
\newtheorem{cor}[thm]{Corollary}
\newtheorem{lem}[thm]{Lemma}
\newtheorem{prop}[thm]{Proposition}

\newtheorem{conj}{Conjecture}
\newtheorem{claim}[thm]{Claim}
\theoremstyle{remark}
\newtheorem{rem}{Remark}

\newcommand{\mfor}{{\ \ \text{for} \ \ }}
\newcommand{\mas}{{\ \ \text{as} \ \ }}

\newcommand{\Id}{\textrm{Id}}

\newcommand{\rad}{\mathrm{rad}}

\renewcommand\Re{\mathrm{Re}\,}

\newcommand{\cl}{\mathcal}
\newcommand{\rest}{\!\!\restriction}

\newcommand{\gs}{Q}

\begin{document}

\title[Equivariant Adkins--Nappi wave maps]{Conditional stable soliton resolution  for a  semi--linear Skyrme
	equation}

\author{Andrew Lawrie, Casey Rodriguez}

\begin{abstract}
We study a semi--linear version of the Skyrme system due to Adkins and Nappi.  The objects in this system are maps from $(1+3)$--dimensional  Minkowski space into the $3$--sphere and 1--forms on $\R^{1+3}$,  coupled via a Lagrangian action.  Under a co--rotational symmetry reduction  we establish the existence, uniqueness, and unconditional asymptotic stability of a family of stationary solutions $Q_n$,  indexed by the topological degree $n \in \N \cup \{0\}$ of the underlying map. We also prove that an arbitrarily large equivariant perturbation of $Q_n$ leads to a globally defined solution that scatters to $Q_n$ in infinite time as long as the critical norm for the solution remains bounded on the maximal interval of existence given by the local Cauchy theory. We remark that the evolution equations are super--critical with respect to the conserved energy. 
\end{abstract}


\maketitle


\section{Introduction}


The goal of this paper is two--fold. First, we will establish the existence, uniqueness, and asymptotic stability of topologically nontrivial stationary solutions for a semi--linear (and co--rotational) version of the Skyrme model introduced in the physics community by Adkins and Nappi~\cite{AN, Skyrme}. Then we will establish stable soliton resolution for this equation, conditional on a certain non--conserved norm remaining bounded over the course of the evolution. We remark that this paper is a natural continuation of~\cite{L15}, which established the same results in the class of topologically trivial maps for the same equation.  Before stating our mains results we first briefly motivate and then introduce the model under consideration.

The Skyrme system and the simplified Adkins--Nappi system considered here are modifications of the $O(3)$ nonlinear $\sigma$-model from particle physics. In the mathematics community, the $O(3)$ nonlinear $\s$-model is called the wave maps equation for maps  $U:(\R^{1+3}, \etab) \to (\Sp^3, \bfg)$, where $\etab$ is the Minkowski metric on $\R^{1+3}$ and $\g$ is the round metric on the $3$-sphere, $\Sp^3$.  A wave map is  a formal critical point of the Lagrangian action 
\EQ{
	\LL(U) = \frac{1}{2} \int_{\R^{1+3}} \etab^{\al \be} \ang{ \p_{\al} U, \p_{\be} U}_{\g} \, d x\, d t.
} 
The Euler--Lagrange equations for  $\LL$ are  
\EQ{ \label{wmi}
	\etab^{\mu \nu} D_{\mu} \p_{\nu} U = 0,
}
where $D$ is the pull-back covariant derivative on the pull--back bundle $U^*T\Sp^3$.   Wave maps conserve the energy, 
\EQ{ \label{en} 
	\E(U, \p_t U)(t) = \frac{1}{2} \int_{\R^3} (\abs{\p_t U}_{\g}^2 + \abs{\nabla U}_{\g}^2) \, dx = \E(U, \p_t U)(0)
}
and also exhibit the following scaling  invariance: for any $\la>0$  
\EQ{
	\vec U(t,x) :=(U(t, x), \p_tU(t, x)) \mapsto  \vec U_{\la}(t,x) := (U(t/ \la, x/ \la), \la^{-1} \p_tU(t/\la, x/ \la)). 
} 
The $O(3)$ wave maps equation  is energy super--critical since one can reduce the energy by concentrating the solution to a point via a rescaling, i.e., sending $\la \to 0$ above, we have 
\EQ{
	\E(\vec U_{\la})  = \la \E(\vec U) \to 0 \mas \la \to 0 
}
making it energetically favorable for the solution to concentrate. One thus expects smooth finite energy initial data to lead to finite--time blow up.  This is evident in the co--rotational formulation of the wave maps equation,  where  the system of equations for $\vec U$ reduce to an equation for $\psi$, the azimuth angle measured from the north pole (in spherical coordinates on $\Sp^3$). That is, \eqref{wmi} reduces to  
\EQ{
	\psi_{tt}- \psi_{rr} - \frac{2}{r} \psi_r  +  \frac{\sin 2\psi}{r^2}  = 0, \quad \vec \psi(0) = (\psi_0, \psi_1)
}
and the conserved energy~\eqref{en}  becomes 
\ant{
	\E(\vec \psi) = \frac{1}{2} \int_0^{\infty} \left(\psi_t^2 + \psi_r^2 + \frac{2\sin^2 \psi}{r^2} \right) \, r^2 \, \ud r.
}
Shatah~\cite{Shatah} proved that such wave maps can form a singularity in finite time. This was made explicit by Turok and Spergel~\cite{TS} who found a closed form  example of self--similar blow--up namely, 
\EQ{ \label{arctan} 
	\psi(t, r)  = 2 \arctan(r/t), 
}
which exhibits derivative blow--up at $t =0$. This singularity formation makes the $O(3)$ wave maps equation unappealing as a physical model. Another significant deficiency  is the absence of finite energy, nontrivial, stationary solutions, harmonic maps in this case, which are often referred to in the physics literature as \emph{topological solitons}~\cite{MS}.

The physics community has come up with several alternatives to the wave maps equation designed to remedy these deficiencies, i.e.,  to remove the possibility of finite--time blow up and to introduce topological solitons. Perhaps the most famous modification is due to Skyrme~\cite{Skyrme}. The Skyrme Lagrangian includes higher order terms to the wave maps Lagrangian that break the scaling symmetry of the resulting system but also give rise to a quasilinear system of equations,  which are difficult to analyze from a dynamical perspective. While it's still unknown if blow up is possible in the full Skyrme equations, the system is known to possess finite energy stationary solutions, called \emph{Skyrmions}. Indeed, the existence of co--rotational Skyrmions was established rigorously by McLeod and Troy~\cite{McT}.

There have been several advances in understanding the dynamical properties of the Skyme system in recent years. For example, the asymptotic stability of Skyrmions was addressed numerically in~\cite{BCR}, and their \emph{linear} stability was established rigorously in~\cite{CDSS}.  Global existence and scattering for initial data that is small in the space $(\dot{B}^{\frac{5}{2}}\times \dot{B}^{\frac{3}{2}}) \cap (\dot{H}^1 \times L^2)$ was proved in~\cite{GNR}, and global existence for large smooth initial data was established in the preprint~\cite{Li}.  

Much stronger results are conjectured in the literature. For example,  \cite{BCR} conjectures that any smooth finite energy, topological degree $n$ initial data leads to a global solution which relaxes (by radiating off its excess energy) to a degree $n$ Skyrmion as $t \to \pm \infty$. In other words, Skyrmions are believed to be globally asymptotically stable in the energy space, i.e., stable soliton resolution holds. The full conjecture presents many significant challenges starting with the fact that the equations, while not scaling invariant, should still be viewed as \emph{super--critical with respect to the conserved energy}. 

The difficulties presented by the Skyrme model lead one to consider even simpler modifications of the wave maps model that retain some of the interesting features of Skyrme, i.e., modifications that remove the scaling instability and that have topological solitons. In~\cite{KLS}, the first author,  Kenig and Schlag, considered a model introduced in~\cite{Biz}, namely  $(1+3)$-dimensional wave maps exterior to the unit ball  taking values in $\Sp^3$. There, it was shown that any finite energy co--rotational data leads to a global and smooth solution which scatters to the unique harmonic map in its degree class as $t \to \pm \infty$, giving the first example of \emph{stable soliton resolution} for a non--integrable equation.   This extended~\cite{LS}, which established the result in the case of topologically trivial initial data and was later followed by~\cite{KLLS1, KLLS2} which extended~\cite{KLS} to data in all higher equivariance classes. Recently the second author expanded on the methods introduced in~\cite{KLS} to prove stable soliton resolution for equivariant wave maps posed on a curved wormhole geometry~\cite{CR16b, CR16c}.  These latter results also used crucially a variant of the linear theory introduced in~\cite{SSS1, SSS2}. The crucial common element in these works is that singularity formation is prevented by removing the only place in the domain where a radially symmetric map can form one, i.e, $r =0$. We also remark that the results described in this paragraph were all inspired by the concentration compactness ideology stemming from the work of Bahouri and G\'erard~\cite{BG} and Kenig and Merle~\cite{KM06, KM08} and the channels of energy technique introduced in Duyckaerts, Kenig, and Merle~\cite{DKM1, DKM2, DKM3, DKM4, DKM5}.

In this paper we consider a different semi--linear modification of the Skyrme system introduced by Adkins and Nappi~\cite{AN} in the mid $1980$'s. Before stating the main results, we provide a brief introduction to the Adkins-Nappi formulation. We refer the reader to the recent works~\cite{GNR, GR10b, GR10a} for excellent introductions including  physical motivation.

\subsection{Adkins Nappi  Maps}
The objects we consider here are maps  $U: (\R^{1+3}, \etab) \to (\Sp^3, \g)$,  where $\etab$ is the Minkowski metric on $\R^{1+3}$ and $\g$ is the usual metric on $\Sp^3$, and $1$--forms (or gauge potentials) $A = A_{\al} dx^{\al}$ over $\R^{1+3}$.  Adkins-Nappi wave maps are a coupled system for the pair $(U, A)$  defined by the requirement that $(U, A)$ are formal critical points for the Lagrangian 
\EQ{ \label{eq:ANL}
	\LL(U, A) &=  \frac{ 1}{2}\int_{R^{1+3}} \etab^{\al \be} \ang{ \p_{\al} U, \p_{\be} U}_\g \, dx\, dt  + \frac{1}{4} \int_{\R^{1+3}} F_{\al \be} F^{\al \be} \, dt \, dx \\
	& \quad - \int_{\R^{1+3}} A_{\al} j^{\al} \, \ud t \, \ud x,
}
where $F_{\al \be} := \p_{\al}A_{\be}- \p_{\be} A_{\al}$ is  the curvature (or Faraday tensor) associated to $A$, and $j$ (called the baryonic current) is given by  
\begin{align*}
j^{\al} = c \eps^{\al \be \ga \de}(\etab) \p_{\be} U^i \p_{\ga} U^j \p_{\de} U^{k} \epsilon_{i j k}(\g).
\end{align*}
Here $c>0$ is a  normalizing constant and $\eps$ is the Levi-Civita symbol,  
\begin{align*}
&\eps_{ijk}(\g) = \sqrt{\abs{\g}}  [ i, j, k],\quad 
\eps^{\al \be \ga \de}(\etab) = \frac{1}{ \sqrt{\abs{ \etab}} }[ \al, \be, \ga, \de]. 
\end{align*}
We remark that ~\eqref{eq:ANL} generalizes the Lagrangian for wave maps from $\R^{1+3}$ taking values in $\Sp^3$ (the case $c = 0$). Following \cite{GNR, GR10b, GR10a, L15}, we consider a restricted class of maps $U$ and $1$--forms $A$ by making an equivariance assumption. Let $(t, r, \om)$  
be polar coordinates on $\R^{1+3}$ with metric $\etab = - dt^2 + dr^2 + r^2 d \om^2$, where $d \om^2$ is the standard metric on $\Sp^2$.  
And let $(\psi, \Om)$ denote spherical coordinates on $\Sp^3$ with metric $\g = d \psi^2 + \sin^2 \psi d \Om^2$ where $d \Om^2$ is the standard metric on $\Sp^2$.  
We consider only co--rotational, or  $1$--equivariant maps, with $\psi = \psi(t, r)$ and $\Om = \om$ so that  $U = U(t, r, \om)$ is given by 
\ant{
	U(t, r, \om) =  ( \psi(t, r), \om). 
}
We also require that the $1$-form $A$ takes the simple form   
\EQ{ \label{eq:AV}
	A(t, r, \om) = ( V(t, r), 0, 0, 0).
}
for some real-valued function $V = V(t, r)$. It follows that 
\ant{
	&\etab^{\al \be} \g_{ij}(U) \p_\al U^i \p_\be U^j = - \psi_t^2 + \psi_r^2 + 2 r^{-2} \sin ^2 \psi, \quad F_{\al \be} F^{\al \be} = - 2 V_{r}^2,\\
	&A_{\al} j^{\al} = V j^0  = 6c V \sin^2 \psi \psi_r r^{-2} = 3c r^{-2}V \p_r( \psi- \frac{1}{2}\sin2 \psi).
}
Under these restrictions on $(U, A)$ the Lagrangian action reduces to 
\ant{ 
	\frac{1}{ \pi} \LL(\psi, V) &= \frac{1}{2}\int_{\R}\int_0^{\infty}  \left( -\psi_t^2 + \psi_r^2 + 2\frac{ \sin^2 \psi}{r^2} \right) \, r^2 \, \ud r \, \ud t\\
	& \quad- \frac{1}{2} \int_{\R} \int_0^{\infty} V_r^2 \, r^2 \, \ud r  \, \ud t- 3c \int_{\R}\int_0^{\infty} V_r ( \psi- \sin \psi \cos \psi) \,\ud r \, \ud t, 
}
which can be  conveniently rewritten as 
\ant{
	\frac{1}{ \pi}   \LL( \psi, V) &= \frac{1}{2} \int_{\R}\int_0^{\infty}  \left( -\psi_t^2 + \psi_r^2 + 2\frac{ \sin^2 \psi}{r^2}  + \al^2 \frac{ (\psi- \sin \psi \cos \psi)^2}{r^4}\right) \, r^2 \ud r \, \ud t \\
	&\quad -  \frac{1}{2}\int_{\R} \int_0^{\infty} \left( r V_r +  \al  \frac{ ( \psi- \sin \psi \cos \psi)}{r} \right)^2\, \ud r \, \ud t, 
}
where the constant $\al$ is defined as $\al = 3c$.  If $( \psi, V)$ is a critical point for $\LL$ then  
\ant{
	0=\frac{d}{d \e} \vert_{\e=0} \LL( \psi, V+ \e W) = -C\int_{ \R} \int_0^{\infty}  \p_r\left(r^2V_r + \al ( \psi- \sin \psi \cos \psi) \right) W\, \ud r \, \ud t.
}
From this we can deduce that 
\EQ{ \label{eq:Veqn}
	r^2V_r + \al  (\psi- \sin \psi \cos \psi) = 0
}
for any critical  point $(\psi, V)$. 
This leads to a decoupling of the Euler-Lagrange equations for $\psi$ and $V$. After rescaling the coordinates~$(t, r) \mapsto ( \al t, \al r)$, we use \eqref{eq:Veqn} to obtain the  equation for $\psi$, which we formulate with Cauchy data, 
\EQ{\label{eq:an}
	&\psi_{tt} -\psi_{rr} -\frac{2}{r} \psi_r + \frac{\sin(2 \psi)}{r^2} + \frac{( \psi - \sin\psi \cos\psi)(1-\cos2 \psi)}{r^4} = 0,\\
	& \vec \psi(0):= (\psi(0), \psi_t(0)) = ( \psi_0, \psi_1).
}
The system has a coercive conserved energy, or Hamiltonian, given by 
\ant{
	\E( \vec \psi(t))= \frac{1}{2} \int_0^{\infty} \left( \psi_t^2 + \psi_r^2 + \frac{2 \sin^2 \psi}{r^2} + \frac{(\psi- \sin\psi \cos\psi)^2}{r^4} \right)\, r^2 \, \ud r.
}
Note that finite energy and continuity of the flow  ensure that for each $t \in I$, where $I \ni 0$ is a time interval on which the solution exists, one has 
\ant{
	\psi(t, 0) = 0, \quad \lim_{r \to \infty} \psi(t, r) =  n \pi
}
for some fixed $ n \in \Z$. We'll refer to $n$ as the \emph{degree} of the map and denote by $\E_{ n}$ the set of all finite energy data of degree $n$, 
\ant{
	\E_n := \{ (\psi_0, \psi_1) \mid \, \E (\vec \psi) <\infty, \, \, \psi_0(0)=0, \, \, \psi_0( \infty) = n \pi\}.
}
We remark that~\eqref{eq:an} is locally well--posed for smooth initial data in~$\E_n$. 

In Section~\ref{s:elliptic} we'll prove that for each $n$,  \emph{ there exists a unique stationary solution} $\vec Q_n = (Q_n, 0) \in \E_n$ to~\eqref{eq:an}; see Theorem~\ref{t01}. The point of this paper is to illustrate the important role that the $Q_n$ play in the dynamics of solutions to~\eqref{eq:an}. In particular, the soliton resolution conjecture for this equation is as follows. 

\begin{conj}[Soliton Resolution] Denote by $\vec Q_n = (Q_n, 0)$ the unique finite enegy stationary  solution to~\eqref{eq:an} in $\E_n$ and  let $(\psi_0, \psi_1) \in \E_n$ be  any smooth compactly supported initial data for~\eqref{eq:an}.  Then, the unique solution $\vec \psi(t)$ to~\eqref{eq:an} with initial data $\vec \psi(0) = (\psi_0, \psi_1)$ is defined globally in time. Moreover, there exists a solution $\vec \fy_L(t)$ to the linear equation 
	\EQ{ \label{eq:linw} 
		\fy_{tt} - \fy_{rr} - \frac{2}{r} \fy_r + \frac{2}{r^2} \fy  = 0
	}
	so that 
	\EQ{
		\vec \psi(t) = \vec Q_n  + \vec \fy_L (t) + o_{\dot{H}^1 \times L^2}(1) \mas t \to \infty.
	}
\end{conj} 

In this paper we prove two partial versions of the conjecture. First, we'll prove that the conjecture is true as long as the initial data in $\E_n$ is close enough to $Q_n$ in a suitable sense. Then, we'll verify the conjecture under the additional assumption that the difference between the nonlinear evolution $\vec \psi(t)$ and $\vec Q_n$ remains uniformly bounded in a sense to be described below. In order to formulate these results precisely,  we first introduce a norm in which to measure proximity to $Q_n$ that  reflects the energy super--critical nature of the Cauchy problem~\eqref{eq:an}. Indeed, given $\vec \psi = (\psi_0, \psi_1) \in \E_n$ we define the notation
\EQ{ 
	\| \vec \psi \|_{\HH_n}^2 := \| ( \psi_0 - Q_n, \psi_1) \|_{( \dot{H}^2 \times \dot{H}^1) \cap (\dot H^1 \times L^2)(\R^3)}^2
} 
and for $n =0$ we have the norm:
\EQ{
	\| \vec \psi \|_{\HH_0}^2 := \| ( \psi_0, \psi_1) \|_{( \dot{H}^2 \times \dot{H}^1) \cap (\dot H^1 \times L^2)(\R^3)}^2.
}
We define $\HH_n \subset \E_n$ as the set 
\EQ{
	\HH_n:= \{ (\psi_0, \psi_1) \in \E_n \, | \,   \|(\psi_0, \psi_1) \|_{\HH_n} < \infty \}.
}
The inclusion of the $\dot{H}^2 \times \dot{H}^1$-norm in these definitions is required to formulate the small data well-posedness  and scattering theory for~\eqref{eq:an} for topologically trivial initial data, see e.g.,~\cite{GNR, L15}, and will also arise in our proof of the asymptotic stability of the non--trivial $Q_n$. It reflects the critical regularity of the underlying scale invariant system and we refer to~\cite{GNR, L15} and Section~\ref{s:SD} for more on this point. 

Before stating the main theorem we refer briefly to Proposition~\ref{p32} below which yields the following: given initial data $(\psi_0, \psi_1) \in \HH_n$ there exists a unique solution $\vec \psi(t) \in \HH_n$ with data $\vec \psi(0)$ defined on a maximal interval of existence $I_{\max}(\vec \psi(0))$. In particular, for each $t \in I_{\max}(\vec \psi(0))$ we can define $\vec \fy(t) \in \HH_0$ by 
\EQ{
	\vec \fy(t):=  \vec \psi(t) - (Q_n, 0).
}

\begin{thm}[Main Theorem] \label{t:main} 
	Let $(\psi_0, \psi_1) \in \HH_n$ for some $n \in \N$ and denote by~$\vec \psi(t)$ the solution to~\eqref{eq:an} with data $\vec \psi(0) = (\psi_0, \psi_1)$  defined on its maximal interval of existence $I_{\max} = I_{\max}(\vec \psi(0))$. The following hold true: 
	\begin{itemize} 
		\item[(a)] For all $n \in \N$, $\vec Q_n=(Q_n, 0)$ is asymptotically stable in $\HH_n$. To be precise: there exists $\eps_0>0$ small enough so that for all data $\vec \psi(0) \in \HH_n$ with 
		\EQ{
			\| \vec \psi(0) \|_{\HH_n}  < \eps_0
		}
		we have $I_{\max}( \vec \psi(0)) = \R$, i.e., the corresponding solution $\psi(t)\in \HH_n$ is defined globally in time. Moreover, $\vec \psi(t)$ scatters to $\vec Q_n$. That is, there exists linear waves $\vec \fy_L^{\pm}$ solving~\eqref{eq:linw} so that 
		\EQ{
			\vec \psi(t) = \vec Q_n  + \vec \fy_L^{\pm} (t) + o_{\HH_0}(1) \mas t \to \pm \infty,
		}
		or in other words,  
		\EQ{
			\|\vec \psi(t) - \vec Q_n  - \vec \fy_L^{\pm} (t) \|_{\HH_0} \to 0 \mas t \to \pm \infty.
		}
		\item[(b)] For all $n \in \N$, large data soliton resolution holds in $\HH_n$ in the following conditional sense. Let $\vec \psi(t)$ be as above and for each $t \in I_{\max}(\vec \psi(0)) := (-T_-, T_+)$ set 
		\EQ{
			\vec \fy(t):= \vec \psi(t) - \vec Q_n
		}
		Suppose that there exists a constant $C< \infty$ for which 
		\EQ{ \label{eq:Hnap} 
			\sup_{t \in [0, T_+)} \| \vec \fy(t) \|_{\HH_0} \le C < \infty
		}
		Then $T_+= + \infty$, i.e., $\vec \psi(t)$ is defined globally in forward time. Moreover, $\vec \psi(t)$ scatters to $Q_n$ in forward time. That is, there exists a  linear wave $\vec \fy_L^+$ solving~\eqref{eq:linw} so that 
		\EQ{
			\vec \psi(t) = \vec Q_n  + \vec \fy_L^{+} (t) + o_{\HH_0}(1) \mas t \to \infty,
		}
		or, in other words, 
		\EQ{
			\|\vec \psi(t) - \vec Q_n  - \vec \fy_L^+ (t) \|_{\HH_0} \to 0 \mas t \to \infty. 
		}
		A similar statement holds in the negative time direction as well. 
	\end{itemize} 
\end{thm}

A few remarks are in order: 
\begin{rem}
	We emphasize that Theorem~\ref{t:main}(b) is a conditional result because it is not known if ~\eqref{eq:Hnap} holds for every solution.  Indeed, it is not clear that higher Sobolev norms such as the $\dot{H}^2 \times \dot{H}^1$ norm of $\vec \psi - \vec Q_n$ should  remain uniformly bounded over the course of the evolution. \end{rem} 

\begin{rem}\label{r:scat}
	We also note that Theorem~\ref{t:main} in the case $n=0$ was proved in~\cite{L15}. The work~\cite{L15} falls into a category of conditional large data scattering results for energy sub--critical and super--critical wave equations in recent years; see e.g., ~\cite{KM10, KM11a, KM11b, DKM5, KV10, KV11b,  L15, DL1, DL2, DR15, CR16a, DY17} and reference therein. In all of these works, one begins with the assumption that the scaling critical Sobolev norm of the solution stays uniformly bounded over the maximal interval if existence for the solution. Indeed, this information is a crucial ingredient in the implementation of the Kenig--Merle concentration compactness/ridigity method, which is used in each of the works cited above. We single out the recent works \cite{DR15} and \cite{DY17} from the list given above, as they do not assume a uniform bound on the critical norm, but rather only that a critical norm remains bounded on a single sequence of times tending to the maximal time of existence. 
\end{rem}

\begin{rem}
	What distinguishes Theorem~\ref{t:main} from the references listed in Remark~\ref{r:scat} is that here we prove that solutions (conditionally) asymptotically relax to a \emph{nontrivial} stationary solution $\vec Q_n$ as opposed to relaxing to the vacuum solution $\vec 0$, i.e., scattering to zero.  To the authors knowledge the present work is the first result in a super--critical setting that establishes conditional large data scattering  to a truly nonlinear object,  $Q_n$.   
\end{rem}

\subsection{Outline of the paper} 

In Section 2, we establish the existence and uniqueness of the stationary solution $Q_n$.  The existence of $Q_n$ follows easily by a variational argument applied to the static energy functional. However, proving $Q_n$ is the unique stationary solution with finite energy is much more subtle.  We use a comparison argument to reduce the proof of uniqueness to showing the following monotonicity property of stationary solutions to \eqref{eq:an}: every finite energy stationary solution $\varphi(r)$ to \eqref{eq:an} satisfies $\varphi'(r) > 0$ $\forall r \geq 0$.  This monotonicity property is proved via a series of simple calculus lemmas along with a comparison argument applied on intervals where $\varphi'(r) \neq 0$.  

In Section 3, we reformulate our main result in terms of the variable $u$ which is related to the original azimuth angle $\psi$ via $\psi(t,r) = Q_n(r) + r u(t,r)$ (see the statement of Theorem \ref{t31}).  This substitution implies that $u$ satisfies a semi--linear equation on $\R^{1+5}$ of the form 
\begin{align*}
u_{tt} - u_{rr} - \frac{4}{r} u_r + V(r) u + \sum_{p = 2}^5 Z_p(r,ru) u^p = 0,
\end{align*}
where the potential $V$ and functions $Z_p$ are smooth and bounded (with $V$ having fast decay).  The energy super--critical nature of the problem is reflected in the fact that the top order nonlinearity in the equation is quintic.  We then prove Strichartz estimates for radial solutions to the linear equation $w_{tt} - \Delta w + V w = F$ on $\R^{1+5}$ which are essential in studying the nonlinear evolution for $u$.  To do this, we first prove linear stability of the radial Schr\"odinger operator $-\Delta + V$ on $\R^5$ using Sturm--Liouville theory (see Theorem \ref{t:lin} for the precise statement of linear stability and Strichartz estimates).  The monotonicity property $Q_n' > 0$ proved in Section 2 is also essential in this argument.  Strichartz estimates then follow from linear stability by using (now standard) arguments which are sketched in Appendix A.   

The proof of our main result begins in earnest in Section 4.  We use the Strichartz estimates developed in Section 3 and standard contraction mapping arguments to prove local well--posedness and small data scattering for \eqref{eq:an} in the equivalent $u$--formulation (see Proposition \ref{p32}).  The latter result is equivalent to the asymptotic stability of $Q_n$ in the original azimuth angle $\psi$ (Theorem \ref{t:main}(a)).  We then prove a long--time stability result which is a technical tool used in the proof of Theorem \ref{t:main}(b) in the final section.  

In Section 5, we prove our conditional stable soliton resolution result, Theorem \ref{t:main}(b), via the Kenig--Merle concentration compactness/rigidity method.  This method has three main steps which we now briefly describe in the azimuth angle $\psi$ (our implementation, however, is in the $u$ variable).  For the first step we show that solutions to \eqref{eq:an} that evolve from small perturbations of $\vec Q_n$ satisfy the conclusions of Theorem \ref{t:main}(b): global existence and scattering to $\vec Q_n$.  This is precisely Theorem \ref{t:main}(a) which was proved in Section 4.  In the second step we show that if the conclusions of Theorem \ref{t:main}(b) are false, then there exists a \emph{critical element}: a minimal solution $\vec \psi_\infty(t)$ to \eqref{eq:an} that fails to scatter to $\vec Q_n$ in both time directions.  Moreover $\vec \psi_\infty(t)$ has the property that the trajectory 
\begin{align*}
\cl K = \{ \vec \psi_\infty(t) \, | \, t \in \R \} \subset \cl H_n 
\end{align*}        
is precompact in $\cl H_n$ (see Proposition \ref{p:cc} for the statement of this step in terms of $u$).  This step follows from the results proved in Section 4 and general concentration compactness arguments which are completely analogous to the degree $n = 0$ argument in \cite{L15}.  Section 5 begins with a brief sketch of this step.  The final step of the Kenig--Merle method (which occupies the bulk of Section 5) is to prove the following rigidity theorem: if $\vec \psi(t)$ is a solution to \eqref{eq:an} with the property that $\cl K = \{ \vec \psi(t) \, | \, t \in \R \} \subset \cl H_n$ is precompact, then $\vec \psi = \vec Q_n$.  This rigidity property contradicts the second step, thus proving that the conclusions of Theorem \ref{t:main}(b) must be true.  The proof of the rigidity result follows the techniques used in \cite{L15} which were introduced in~\cite{KLS}. This type of argument, based on the ``channels of energy" method introduced in the seminal works~\cite{DKM1, DKM2, DKM3} and especially in~\cite{DKM4, DKM5}, avoids using any dynamical identities or inequalities, such as virial or Morawetz, which are in general poorly suited to the complicated nonlinearities in geometric equations.  The method used here relies on exterior energy estimates for the underlying free radial wave equation proved in~\cite{KLS} (see Proposition~\ref{linear prop}) and the uniqueness of $\vec Q_n$ as the only finite energy stationary solution to \eqref{eq:an} proved in Section 2. 


\section{Stationary Adkins--Nappi  Maps}


\label{s:elliptic}

For the remainder of the paper, we fix a topological degree $n \in \N$.  
In this section, we will establish the existence, uniqueness, and properties of a stationary solution to~\eqref{eq:an}, that is, a solution to the elliptic equation 
\begin{align}
\begin{split}\label{e00}
&\vp_{rr} + \frac{2}{r} \vp_r = \frac{\sin 2\vp}{r^2} + \frac{(\vp - \sin \vp \cos \vp)(1 - \cos 2 \vp)}{r^4}, \quad r > 0, \\
&\vp(0) = 0, \quad \vp(\infty) = n\pi. 
\end{split}
\end{align}
More precisely, we prove the following. 

\begin{thm}[Existence and Uniqueness of Stationary Solutions to~\eqref{eq:an}]\label{t01}
	There exists a unique solution $Q_n$ to \eqref{e00} which we will refer to as the stationary Adkins--Nappi map of degree $n$.  Moreover, $Q'_n(r) > 0$ for all $r \geq 0$ 
	and there exist unique $\al_n, \beta_n > 0$ such that 
	\begin{align*}
	Q_n(r) &= n\pi - \al_n r^{-2} + O(r^{-6}) \quad \mbox{as } r \rar \infty, \\
	Q_n(r) &= \beta_n r + O(r^{3}) \quad \mbox{as } r \rar 0
	\end{align*}
	where the $O(\cdot)$ terms satisfy the natural derivative bounds. 
\end{thm}

\subsection{Existence of Stationary Adkins--Nappi maps} 

We first establish existence by minimizing the static energy functional.  
In what follows, we denote the \emph{static Adkins--Nappi energy} by
\begin{align*}
\cl J(\varphi) = \frac{1}{2} \int_0^\infty \left [ 
(\varphi')^2 + \frac{2 \sin^2 \varphi}{r^2} + \frac{(\varphi - \sin \varphi \cos \varphi)^2}{r^4}
\right ] r^2 dr,
\end{align*}
and the \emph{local Adkins--Nappi energy} by
\begin{align*}
\cl J^{\rho_1}_{\rho_0} (\varphi) = \frac{1}{2} \int_{\rho_0}^{\rho_1} \left [ 
(\varphi')^2 + \frac{2 \sin^2 \varphi}{r^2} + \frac{(\varphi - \sin \varphi \cos \varphi)^2}{r^4}
\right ] r^2 dr.
\end{align*}
We define 
\begin{align*}
X := \Bigl \{ \varphi :\:\: &\varphi \in C([0,\infty)), \quad \varphi \in AC(J) \quad \forall J \Subset (0,\infty), \\
&\varphi(0) = 0, \varphi(\infty) = n\pi, \quad \cl J(\varphi) < \infty \Bigr \}  
\end{align*}
and 
\begin{align*}
\cl J_m := \inf_{\vphi \in X} \cl J(\varphi) 
\end{align*}
which exists since $\cl J(\cdot)$ is nonnegative.

The main goal of this subsection is to prove that $\cl J_m$ is attained in $X$.  

\begin{prop}\label{p01}
	There exists $\vphi \in X$ such that $\cl J_m = \cl J(\vphi)$.  
\end{prop}

Since \eqref{e00} is the Euler--Lagrange equation associated to the functional $\cl J(\cdot)$, an immediate corollary of Proposition \ref{p01} is the existence statement made in Theorem \ref{t01}.  

\begin{cor}\label{c02}
	There exists a smooth stationary Adkins--Nappi map $Q_n$. 
\end{cor}

\begin{proof}[Proof of Proposition \ref{p01}]
	
	Let $(\vphi_k)$ be a minimizing sequence for $\cl J_m$. Since $$\cl J(|\varphi_k|) = \cl J(\varphi_k)$$ we may assume that 
	$\varphi_k \geq 0$ on $[0,\infty)$ for all $k$.  We first show that the sequence is pointwise bounded uniformly in $k$.  
	
	\begin{claim}\label{l01}
		There exists $C > 0$ independent of $k$ such that 
		\begin{align}\label{e01}
		|\varphi_k(r)| \leq C, \quad \forall r > 0.
		\end{align}
	\end{claim}
	
	\begin{proof}[Proof of Claim \ref{l01}]
		We introduce 
		\begin{align*}
		G(\rho) := \int_0^\rho (\theta - \sin \theta \cos \theta) d\theta = \frac{1}{2} ( \rho^2 - \sin^2 \rho ). 
		\end{align*}
		Note that $G(0) = 0$, $G$ is even and increasing on $[0,\infty)$, and $G(\rho) \rar \infty$ as $\rho \rar \infty$.  Then,
		for each $k$, by the fundamental theorem of calculus
		\begin{align*}
		G(\varphi_k(r)) &= G(\varphi_k(r)) - G(\varphi_k(0)) \\
		&= \int_0^r \frac{d}{d\rho} G(\varphi_k(\rho)) d\rho \\
		&= \int_0^r \vphi_k' (\vphi_k - \sin \vphi_k \cos \vphi_k) d\rho \\
		&\leq \left ( \int_0^r (\vphi_k')^2 \rho^2 d\rho \right )^{1/2} 
		\left ( \int_0^r (\vphi_k - \cos \vphi_k \sin \vphi_k )^2 \rho^{-2} d\rho \right )^{1/2} \\
		&\leq 2 \cl J(\vphi_k) \leq M.  
		\end{align*}
		This proves the claim. 
	\end{proof}
	
	From Claim \ref{l01}, we deduce that for each fixed $0 < \rho_0 < \rho_1< \infty$, $(\varphi_k)$ is a bounded sequence in $H^1(\rho_0,\rho_1)$. 
	By Arzela--Ascoli, passing to subsequences, and relabeling, there exists $\vphi \in C(0,\infty)$ such that for all 
	$0 < \rho_0 < \rho_1 < \infty$,
	\begin{align*}
	\vphi_k \rightharpoonup_k \vphi \quad \mbox{weakly in } H^1(\rho_0,\rho_1), \\
	\vphi_k  \rar_k \vphi \quad \mbox{strongly in } C([\rho_0,\rho_1]).  
	\end{align*}
	Hence, we have that 
	\begin{align}
	\begin{split}\label{jm}
	\cl J(\vphi) &= \lim_{\e \rar 0} \cl J_{\e}^{1/\e}(\vphi) \\
	&\leq \limsup_{\e \rar 0} \liminf_k \cl J_{\e}^{\e^{-1}} (\varphi_k) \\
	&\leq \cl J_m.  
	\end{split}
	\end{align}
	
	The proof of Proposition \ref{p01} then follows from showing 
	\begin{claim}\label{l02}
		$\vphi$ is in $X$.  
	\end{claim}
	
	\begin{proof}[Proof of Claim \ref{l02}]
		Let $k \geq 1$ and $r > 0$.  By the fundamental theorem of calculus and Cauchy--Schwarz 
		\begin{align*}
		|n \pi - \vphi_k(r)| &= \left | \int_r^\infty \vphi_k'(\rho) d\rho \right | \\
		&\leq r^{-1/2} \left ( \int_r^\infty (\vphi_k'(\rho))^2 \rho^2 d\rho \right )^{1/2} \\
		&\leq r^{-1/2} (\cl J(\vphi_k))^{1/2} \\
		&\lesssim r^{-1/2}
		\end{align*}
		where the implied constant is uniform in $k$.  We let $k \rar \infty$ and deduce 
		\begin{align*}
		|n \pi - \vphi(r) | \lesssim r^{-1/2}
		\end{align*}
		which proves $\lim_{r \rar \infty} \vphi(r) = n\pi$.  
		
		We now prove $\vphi(0) = 0$.  Let $0 < r_0 < r_1$.  Then 
		as in Lemma \ref{l01}, we have 
		\begin{align*}
		|G(\vphi(r_1)) - G(\vphi(r_0))| \leq 2 \cl  J_{r_0}^{r_1}(\vphi). 
		\end{align*}
		Thus, $\lim_{r \rar 0} G(\vphi(r)) = L$ exists.  Since $G$ is continuously invertible on $[0,\infty)$ and $\vphi \geq 0$ (because 
		$\vphi_k \geq 0, \forall k$), it follows that $\vphi$ has 
		a limit $\ell$ as $r \rar 0$.  If this limit is nonzero, then $\varphi - \sin \vphi \cos \vphi > \ell/2 > 0$ for $r$ near 0 which 
		implies that $\cl J(\vphi) = \infty$, a contradiction.  Thus, $\ell = 0$, and
		we conclude that $\vphi \in X$.    
	\end{proof}
	
	By our definition of $\cl J_m$, \eqref{jm} and Claim \ref{l02}, we have proven Proposition \ref{p01}.  
\end{proof}

\subsection{Uniqueness of stationary Adkins--Nappi maps}

We will now prove the uniqueness of the solution obtained in the previous subsection and establish the properties stated in 
Theorem \ref{t01}.  We will first prove a few facts about solutions to \eqref{e00}.  The first two were proven in similar settings and the proofs follow in almost identical fashion.  

\begin{lem}\label{l05a}
	Let $\vphi$ be a solution to \eqref{e00}.  Then there exist $\alpha, \beta \in \R$ such that 
	\begin{align*}
	\vphi(r) &= n\pi - \al r^{-2} + O(r^{-6}) \quad \mbox{as } r \rar \infty, \\
	\vphi(r) &= \beta r + O(r^{3}) \quad \mbox{as } r \rar 0,
	\end{align*}
	where the $O(\cdot)$ terms satisfy the natural derivative bounds. 
\end{lem}

\begin{proof}
	The proof follows in almost exactly the same way as the proof of Theorem 2.3 in \cite{McT} in the setting of co--rotational Skyrmions.  For the asymptotics at $\infty$, one first shows that $\phi'(r) = o(r^{-1})$.  The idea now is to 
	make the change of variables $x = \log r$ and use the fact that $d\vphi/dx = rd\vphi/dr = o(1)$ to write \eqref{e00} 
	as 
	\begin{align}\label{asympt ode}
	\frac{d^2 \vp}{dx^2} + \frac{d\vp}{dx} - \sin 2\vp + O(e^{-2x} ) = 0.  
	\end{align}
	The ODE \eqref{asympt ode} is asymptotically the autonomous ODE $\vp'' + \vp' - \sin 2\vp = 0$ (the damped pendulum) for which 
	the desired expansion at $\infty$ in the statement of Lemma \ref{l05a} holds (in the $x$ variable).  A similar argument also applies for the asymptotics at $r = 0$.  We omit the details 
	and refer the reader to the proof of Theorem 2.3 in \cite{McT} for the full details of the argument. 
\end{proof}

\begin{lem}\label{l05b}
	Let $\alpha \in \R$.  Then there exists a unique solution $\vp_\al$ to the ODE
	\begin{align*}
	&\vp_{rr} + \frac{2}{r} \vp_r = \frac{\sin 2\vp}{r^2} + \frac{(\vp - \sin \vp \cos \vp)(1 - \cos 2 \vp)}{r^4}, \quad r > 0,
	\end{align*}
	with the property that 
	\begin{align*}
	\vp_\al(r) = n\pi - \al r^{-2} + O(r^{-6}) \quad \mbox{as } r \rar \infty,
	\end{align*}
	where the $O(\cdot)$ term satisfies the natural derivative bounds. 
\end{lem}

\begin{proof}
	See Proposition 2.2 in \cite{L15} for the $n = 0$ case.  The case $n > 0$  is nearly identical and we omit the proof.  
\end{proof}

\begin{lem}\label{l05}
	Let $\vphi$ be a solution to \eqref{e00}.  Then $0 < \vphi(r) < n\pi$ for all $r > 0$.    
\end{lem}

\begin{proof}
	Define 
	\begin{align*}
	W(r) = r^4 \vphi'(r)^2 - 2r^2 \sin^2 \vphi - (\vphi - \sin \vphi \cos \vphi)^2.  
	\end{align*}
	Then using \eqref{e00}, we have that 
	\begin{align*}
	W'(r) = -4r \sin^2 \vphi,
	\end{align*}
	i.e. $W(r)$ is strictly decreasing on $[0,\infty)$.  Moreover, by Lemma \ref{l05a}, we have 
	\begin{align}\label{e05}
	W(0) = 0, \quad W(\infty) = - n^2 \pi^2. 
	\end{align}
	Suppose that $r_0 > 0$ and $\vphi(r_0) = 0$.  Then $W(r_0) > 0$ since $\vphi$ is nonconstant.  However, this is impossible since 
	$W(0) = 0$ and $W(r)$ is decreasing.  Suppose now that there exists $r_0 > 0$ such that $\vphi(r_0) = n\pi$ and
	$\vphi(r) < n\pi$ for all $r < r_0$.  Then
	since $\vphi$ is nonconstant and $\vphi(\infty) = n\pi$, there exists $r_1 > r_0$ such that 
	\begin{align*}
	\vphi(r_1) > n\pi \quad \mbox{and} \quad \vphi'(r_1) = 0.
	\end{align*}
	Since $\vphi \mapsto \vphi - \sin \vphi \cos \vphi$ is an increasing function on $[0,\infty)$, by our definition of $r_1$ we see that   
	\begin{align*}
	W(r_1) &= -2 r_1^2 \sin^2 \vphi(r_1) - (\vphi(r_1) - \sin \vphi(r_1) \cos \vphi(r_1))^2 \\
	&< -( n\pi - \sin n\pi \cos n\pi )^2 \\
	&= - n^2 \pi^2.  
	\end{align*}
	But since $W(r)$ is strictly decreasing, this implies that we cannot have $W(\infty) = -n^2 \pi^2$, a contradiction to \eqref{e05}. Thus, 
	$0 < \vphi(r) < n\pi$ for all $r > 0$. 
\end{proof}

We will now show that two solutions to \eqref{e00} are equal if they both have positive derivatives. The proof is in similar spirit to the argument used to prove uniqueness of corotational Skyrmions in~\cite{McT}.  We will later use ideas from the proof to show that any solution to \eqref{e00} must have positive derivative and conclude the uniqueness of the degree $n$
Adkins--Nappi harmonic map.

\begin{lem}\label{l06}
	Let $\vphi_1,\vphi_2$ be solutions to \eqref{e00} such that $\vphi_j'(r) > 0$ for all $r \geq 0$, $j = 1,2$.  Then $\vphi
	_1 = \vphi_2.$
\end{lem}

\begin{proof}
	We will assume that $n > 0$ since uniqueness in the case $n = 0$ was proved in 
	\cite{L15}.  We first make some preliminary observations.  If $\vphi$ satisfies \eqref{e00}, then by the change of variables $x = \log r$, we see that $\vphi$ satisfies
	\begin{align}
	\begin{split}\label{e06}
	&\vphi'' + \vphi' = \sin 2 \vphi + e^{-2x} (\vphi - \sin \vphi \cos \vphi)(1 - \cos 2 \vphi), \\
	&\vphi(-\infty) = 0, \quad \vphi(\infty) = n\pi.  
	\end{split}
	\end{align}
	Here we now denote $\vphi' = \frac{d\vphi}{dx} = r \frac{d\vphi}{dr}$. 
	If $\vphi' > 0$, we may make a change of variables and consider 
	$\vphi$ as the independent variable and $x = \vp^{-1}$ and $p = \frac{d\vphi}{dx}$ as the dependent variables.  Thus, by \eqref{e06} the equation solved
	by $p$ is 
	\begin{align}
	p \frac{dp}{d\vp} + p = \sin 2 \vphi + e^{-2x} (\vphi - \sin \vphi \cos \vphi)(1 - \cos 2 \vphi). 
	\end{align}
	Here $x = x(\vphi)$.  Suppose, towards a contradiction, that we have two different solutions $\vphi_1, \vphi_2$.  We let 
	$x_j = \vp_j^{-1}$ and $p_j = \frac{d \vp_j}{dx}$.  Then we have
	\begin{align}
	p_j \frac{dp_j}{d\vphi} +  p_j = \sin 2 \vphi + e^{-2x_j} (\vphi - \sin \vphi \cos \vphi)(1 - \cos 2 \vphi), \quad j = 1,2.  
	\end{align}
	Subtracting the equation satisfied by $p_1$ from the equation satisfied by $p_2$
	and rearranging, we have
	\begin{align*}
	p_2 \frac{d}{d\vp}(p_2 - p_1) &+ \Bigl (1 + \frac{dp_1}{d\vphi} \Bigr )(p_2 - p_1) \\&=
	(e^{-2x_2} - e^{-2x_1}) (\vphi - \sin \vphi \cos \vphi)(1 - \cos 2 \vphi)
	\end{align*}
	Let $\vp_0 \in (0,n\pi)$ be arbitrary and define  
	\begin{align*}
	q(\vphi) &= p_2^{-1} \left ( 1 + \frac{dp_1}{d\vphi} \right ),\\ 
	Q(\vphi) &= \int_\vphi^{\vphi_0} q( \phi) d \phi, \\
	f(\vphi) &= p_2^{-1} (e^{-2x_1} - e^{-2x_2}) (\vphi - \sin \vphi \cos \vphi)(1 - \cos 2 \vphi).
	\end{align*} 
	Then  
	\begin{align*}
	\frac{d}{d\vp}(p_2 - p_1) + q (p_2 - p_1) = -f \implies 
	\frac{d}{d\vphi}((p_2 - p_1) e^{-Q} ) = -e^{-Q} f.
	\end{align*}
	Hence, we see that 
	\begin{align}
	(p_2 - p_1)(\vphi) = e^{Q(\vphi)} (p_2 - p_1)(\vphi_0) + \int_\vphi^{\vphi_0} e^{Q(\vphi)-Q(\phi)} f(\phi) d\phi. \label{e07}
	\end{align}

	We now make a key observation based on \eqref{e07}. In particular,
	if $p_2 > 0$ on $(0,n\pi)$ and there exists $ 
	\vp_0 \in (0,n\pi)$ such that $p_2(\vphi_0) > p_1(\vphi_0)$ and $x_2(\vp_0) > x_1(\vp_0)$, then
	\begin{align}
	\begin{split}\label{e08}
	p_2(\vphi) > p_1(\vphi) \mbox{ and } x_2(\vphi) > x_1(\vphi),  \quad \forall
	\vphi \leq \vphi_0.
	\end{split}
	\end{align}
	Indeed, suppose $\vp_1 < \vp_0$ and $p_2(\vp) > p_1(\vp)$ for all $\vp_1 < \vp \leq \vp_0$.  
	Then since $p_j = ( \frac{dx_j}{d\vp} )^{-1}$, we have for 
	all $\vp_1 < \vp \leq \vp_0$ 
	\begin{align*}
	p_2(\vp) > p_1(\vp) &\implies \frac{d}{d\vp}(x_2(\vp) - x_1(\vp)) < 0.
	\end{align*}
	This implies upon integrating that 
	\begin{align*}
	x_2(\vphi) - x_1(\vphi) > x_2(\vphi_0) - x_1(\vphi_0) > 0, \quad \forall \vp_1 \leq \vp \leq \vp_0. 
	\end{align*}
	Since $p_2 > 0$, the previous implies that   
	\begin{align*}
	x_2(\vphi) > x_1(\vphi), \quad \vp_1 \leq \vp \leq \vp_0 &\implies 
	e^{-2x_1(\vphi)} - e^{-2x_2(\vphi)} > 0, \quad \vp_1 \leq \vp \leq \vp_0, \\
	&\implies f(\vp) > 0, \quad \vp_1 \leq \vp \leq \vp_0. 
	\end{align*}
	Hence by \eqref{e07}
	\begin{align*}
	p_2(\vp) - p_1(\vp) > 0, \quad \vp_1 \leq \vp \leq \vp_0.  
	\end{align*}
	Thus, if  $x_2(\vp) > x_1(\vp)$ and $p_2(\vp) > p_1(\vp)$ for all $\vp_1 < \vp \leq \vp_0$, we have $x_2(\vp) > x_1(\vp)$ and
	$p_2(\vp) > p_1(\vp)$ 
	for all $\vp_1 \leq \vp \leq \vp_0$. By a continuity argument it follows that $p_2(\vp) > p_1(\vp)$ and $x_2(\vp) > x_1(\vp)$ for all $\vp \leq 
	\vp_0$.
	
	By Lemma \ref{l05a} and Lemma \ref{l05}, if $\vp$ is a solution to \eqref{e00} then there exist unique $\al, \beta > 0$ such that 
	\begin{align*}
	\vp(x) &= n\pi - \al e^{-2x} + \frac{1}{14} \Bigl ( \frac{2}{3}\al^3 + n\pi \al^2 \Bigr ) e^{-6x} + O(e^{-10x}), \quad \mbox{as } x \rar \infty, \\
	\vp(x) &= \beta e^{x} + \frac{2}{15} ( \beta^5 - \beta^3 ) e^{3x} + O(e^{5x}), \quad \mbox{as } x \rar -\infty,
	\end{align*}
	where the $O(\cdot)$ terms are determined by $\alpha$ and $\beta$ respectively. 
	It follows that $p$ satisfies 
	\begin{align}
	\begin{split}\label{e09}
	p &= 2\al e^{-2x} - \frac{3}{7} \Bigl ( \frac{2}{3}\al^3 + n\pi \al^2 \Bigr ) e^{-6x} + O(e^{-10x})\\
	&= 2( n\pi - \vphi) - \frac{2}{7} \Bigl ( \frac{2}{3}\al^3 + n\pi \al^2 \Bigr ) e^{-6x} + O(e^{-10x})\\
	&= 2 (n\pi - \vp) - \frac{2}{7} \Bigl (\frac{2}{3} + \frac{n\pi}{\al} \Bigr ) (n\pi - \vp)^3 + O((n\pi - \vp)^5)
	\end{split}
	\end{align}
	as $\vp \rightarrow n\pi^-$.  Similarly, we have
	\begin{align}
	p &= 2 \vp + \frac{4}{15} ( \beta^2 - 1 ) \vp^3 + O(\vp^5)  \label{e10}
	\end{align}
	as $\vp \rightarrow 0^+$. Suppose $\vp_2$ has coefficients $\al_2,\beta_2 > 0$ and $\vp_1$ has coefficients $\al_1,\beta_1 > 0$ appearing in their respective asymptotics where 
	(without loss of generality)
	\begin{align*}
	\al_2 > \al_1.  
	\end{align*}
	Then clearly $x_2(\vp) > x_1(\vp)$ for all $\vp$ sufficiently close to $n\pi$ since for $x$ large 
	\begin{align*}
	\vp_2(x) = n\pi - \al_2 e^{-2x} + O(e^{-6x}) < n \pi - \al_1 e^{-2x} + O(e^{-6x}) = \vp_1(x). 
	\end{align*}
	Moreover, we see that $p_2(\vp) > p_1(\vp)$ for $\vp$ sufficiently close to $n\pi$ since by \eqref{e09} we have
	\begin{align*}
	p_2(\vp) &=2 (n\pi - \vp) - \frac{2}{7} \Bigl (\frac{2}{3} + \frac{n\pi}{\al_2} \Bigr ) (n\pi - \vp)^3 + O((n\pi - \vp)^5)\\&> 
	2 (n\pi - \vp) - \frac{2}{7} \Bigl (\frac{2}{3} + \frac{n\pi}{\al_1} \Bigr ) (n\pi - \vp)^3 + O((n\pi - \vp)^5)= p_1(\vp). 
	\end{align*}
	Thus, by our observation \eqref{e08},
	we have $p_2(\vp) > p_1(\vp)$ and $x_2(\vp) > x_1(\vp)$ for all $\vp \in (0,n\pi)$.  In particular, the constraint $x_2(
	\vp) > x_1(\vp)$ for
	all $\vp \in (0,n\pi)$ 
	implies that $\beta_1 > \beta_2 > 0$. But then for $\vp$ near 0, we have by \eqref{e10}
	\begin{align*}
	p_1(\vp) &= 2 \vp + \frac{4}{15} ( \beta_1^2 - 1 ) \vp^3 + O(\vp^5)\\& > 2 \vp + \frac{4}{15} ( \beta_2^2 - 1 ) \vp^3 + O(\vp^5) = p_2(\vp),  
	\end{align*}
	which contradicts $p_2(\vp) > p_1(\vp)$ for all $\vp \in (0,n\pi)$. Thus, no two distinct solutions $\vp_1,\vp_2$ exist.
	This completes the proof. 
\end{proof} 

Before showing that every solution $\vphi(r)$ to \eqref{e00} satisfies 
$\vphi'(r) >  0$ for all $r \geq 0$, we will first need to establish a few simple facts.  

\begin{lem}\label{l07a}
	Let $\vphi$ be a solution to \eqref{e00}, and suppose that $\vphi'(r_0) = \vp''(r_0) = 0$.  Then there exists 
	$\delta > 0$ such that $\vp'(r) < 0$ for all $r$ with $0 < |r - r_0| < \de$. 
\end{lem}

\begin{proof}
	Let $\vphi_0 = \vp(r_0)$. By \eqref{e00} and our assumptions, we have 
	\begin{align}\label{e04}
	0 = \frac{\sin 2 \vphi_0}{r_0^2} + \frac{(\varphi_0 - \sin \vphi_0 \cos \vphi_0)(1 - \cos 2 \vphi_0)}{r_0^4}.
	\end{align}
	From \eqref{e04}, we conclude that 
	\begin{align*}
	r_0^{-2} = -\sin 2 \vphi_0 (\varphi_0 - \sin \vphi_0 \cos \vphi_0)^{-1}(1 - \cos 2 \vphi_0)^{-1}. 
	\end{align*}
	If we differentiate \eqref{e00} and evaluate at $r_0$, we conclude that 
	\begin{align*}
	\vp'''(r_0) &= -\frac{1}{r_0} 
	\Bigl ( \frac{2\sin 2 \vphi_0}{r_0^2} + \frac{4(\varphi_0 - \sin \vphi_0 \cos \vphi_0)(1 - \cos 2 \vphi_0)}{r_0^4} \Bigr ) \\
	&= -\frac{1}{r_0} 
	\Bigl ( -2 (\sin 2 \vphi_0)^2(\varphi_0 - \sin \vphi_0 \cos \vphi_0)^{-1}(1 - \cos 2 \vphi_0)^{-1} \\ &\:\:+ 
	4 (\sin 2 \vphi_0)^2(\varphi_0 - \sin \vphi_0 \cos \vphi_0)^{-1}(1 - \cos 2 \vphi_0)^{-1} \Bigr ) \\
	&= -2 r_0^{-1} (\sin 2 \vphi_0)^2(\varphi_0 - \sin \vphi_0 \cos \vphi_0)^{-1}(1 - \cos 2 \vphi_0)^{-1} \\
	&< 0.  
	\end{align*}
	The final inequality above follows from the fact that $\vphi$ is nonconstant.  Since $\vp''(r_0) = 0$, we conclude that $\vp''(r) > 0$ for $r < r_0$, close to $r_0$, and 
	$\vp''(r) < 0$ for $r > r_0$, close to $r_0$.  This implies that there exists 
	$\delta > 0$ such that $\vp'(r) < 0$ for all $r$ with $0 < |r - r_0| < \de$. 
\end{proof}

\begin{cor}\label{c07b}
	Let $\vp$ be a solution to \eqref{e00}, and suppose that $\vp'(r_0) = 0$.  Then $\vp$ cannot be nondecreasing in a neighborhood 
	of $r_0$.  
\end{cor}

\begin{proof}
	If our conclusion was false, then the assumed monotonicity of $\vp$ and $\vp'(r_0) = 0$ imply that $\vp''(r_0) = 0$.  By Lemma 
	\ref{l07a}, we conclude that $\vp'(r)< 0$ for all $r > r_0$ close to $r_0$, a contradiction.  
\end{proof}

\begin{lem}\label{l03}
	Let $\vphi$ be a solution to \eqref{e00}.  Then the set of critical points of $\vphi$ is finite.   
\end{lem}

\begin{proof}
	We denote the set of critical points of $\vp$ by  
	$C = \{ r \geq 0 : \vp'(r) = 0 \}$.  By Lemma \ref{l05a} and Lemma \ref{l05} there 
	exists $R > 0$ large so that
	\begin{align*} 
	\vp'(r) > 0, \quad \forall r \in [0,R^{-1}] \cup [R,\infty).
	\end{align*}
	Hence $C$ is a compact set.  If $C$ is infinite, then there exists
	a monotonic increasing or decreasing sequence $(r_n) \subseteq C$ and $r_0 \in C$ such that $r_n \rar_n r_0$.  By the 
	mean value theorem and continuity this implies that at $r_0$
	\begin{align*}
	\vphi'(r_0) = \vphi''(r_0) = 0.  
	\end{align*}
	By Lemma \ref{l07a} there exists 
	$\delta > 0$ such that $C \cap (r_0 - \delta, r_0 + \delta) = \{ r_0 \}$, a contradiction to the definition of $r_0$.  Thus, 
	$C$ is finite.  
\end{proof}

We will now show that every solution $\vphi$ to \eqref{e00} has positive 
derivative on $[0,\infty)$.

\begin{lem}\label{l07}
	Let $\vphi$ be a solution to \eqref{e00}.  Then $\vphi'(r) > 0$ for all $r \geq 0$.  
\end{lem}

\begin{proof}
	As in the proof of Lemma \ref{l06} we first make the change of variables $x = \log r$, so that $\vp$ satisfies
	\begin{align*}
	&\vphi'' + \vphi' = \sin 2 \vphi + e^{-2x} (\vphi - \sin \vphi \cos \vphi)(1 - \cos 2 \vphi), \\
	&\vphi(-\infty) = 0, \quad \vphi(\infty) = n\pi,  
	\end{align*}
	where $\vp' = \frac{d\vp}{dx}$.  Suppose that our conclusion is false and that there exists $x_0 \in \R$ such that 
	$\vp'(x_0) = 0$.  By Lemma \ref{l03}, the set 
	of critical points $C = \{ x \in \R : \vp'(x) = 0 \}$ is finite and nonempty by assumption.  We label these points by 
	\begin{align*}
	x_N < \cdots < x_2 < x_1.
	\end{align*}
	We denote the value of $\vp$ at these points by $\vp_j = \vp(x_j)$.  Note that by Lemma \ref{l05}, we have 
	$\vp_j \in (0,n\pi)$ for all $1 \leq j \leq N$.  
	
	We claim that 
	\begin{align}\label{e11}
	\forall x\in [x_N,x_1], \quad \vp(x) \geq \vp_1.
	\end{align}
	In fact, we will prove that for all $1 \leq j \leq N-1$ 
	\begin{align}\label{e12b}
	\forall x \in (x_{j+1}, x_j), \quad \vp(x) > \vp_1. 
	\end{align}
	The proof is by induction on $j$.  
	By Lemma \ref{l05} and our definition of $x_1$, we see that 
	\begin{align}\label{e13}
	\vp \mbox{ is increasing on } (x_1,\infty). 
	\end{align}
	By Lemma \ref{c07b}, we must have that $\vp'(x) < 0$ for all $x \in (x_2,x_1)$.  Hence $\vp(x) > \vp_1$ for all 
	$x \in (x_2,x_1)$ which proves the base case. Suppose that $\eqref{e11}$ holds for all $1 \leq j \leq N-2$.  We wish to show that 
	\eqref{e12b} holds for $j + 1$.
	There are two cases to 
	consider: $\vp'(x) < 0$ on $(x_{j+2}, x_{j+1})$, or $\vp'(x) > 0$ on $(x_{j+2},x_{j+1})$. 
	
	\emph{Case 1:} $\vp'(x) < 0$ on $(x_{j+2}, x_{j+1})$.  By our induction hypothesis and the fact that 
	$\vp$ is strictly decreasing on $[x_{j+2},x_{j+1})$ we conclude that 
	\begin{align}\label{e12}
	\quad \forall x \in [x_{j+1}, x_j), \quad \vp(x) > \vp_{j+1} \geq \vp_1.
	\end{align}
	This concludes the argument for this case. 
	
	\emph{Case 2:} $\vp'(x) > 0$ on $(x_{j+2}, x_{j+1})$.  For this case we use arguments from Lemma \ref{l06}. We first note that 
	by Lemma \ref{c07b}, we must have that $\vp'(x) < 0$ on $(x_{j+1}, x_{j})$.  In particular, by our induction 
	hypothesis and \eqref{e12}, 
	we have that 
	\begin{align*}   
	\vp_{j+1} > \vp_1.
	\end{align*}
	Suppose, towards a contradiction, that $\vp_{j+2} < \vp_1$.  Let $a \in (x_{j+2}, x_{j+1})$ be such that 
	$\vp(a) = \vp_1$, and denote $b = x_{j+1}$.  Let $d \in (x_1, \infty)$ be such that 
	$\vp(d) = \vp_{j+1}$, and denote $c = x_1$.    By the previous definitions, we have 
	that 
	\begin{align}
	\begin{split}\label{e14}
	&\forall x \in [a,b) \cup (c,d), \quad \vp'(x) > 0,  \\
	&0 = \vp'(b) < \vp'(d), \\
	&a< b < c < d. 
	\end{split}
	\end{align}
	
	We now consider $\vp \in (\vp_0,\vp_{j+1})$ as the dependent variable and define two functions
	\begin{align*}
	x &= \bigl ( \vp|_{x \in (a,b)}\bigr )^{-1} : (\vp_1, \vp_{j+1}) \rar (a,b), \\
	y &= \bigl (\vp|_{y \in (c,d)}\bigr )^{-1} : (\vp_1, \vp_{j+1}) \rar (c, d).
	\end{align*}
	Let 
	\begin{align*}
	p = \Bigl (\frac{dx}{d\vp} \Bigr )^{-1} = \frac{d\vp}{dx} \Bigl |_{x \in (a,b)}, \\
	q = \Bigl (\frac{dy}{d\vp} \Bigr )^{-1} = \frac{d\vp}{dy} \Bigl |_{y \in (c,d)}. 
	\end{align*}
	Then \eqref{e14} implies that  
	\begin{align}
	\begin{split} \label{e15}
	&\forall \vp \in (\vp_1,\vp_{j+1}), \quad p,q > 0, \\
	&\mbox{$\forall \vp < \vp_{j+1}$ close to $\vp_{j+1}$}, \quad p(\vp) < q(\vp), \\
	&\forall \vp \in (\vp_1, \vp_{j+1}), \quad x(\vp) < y(\vp).
	\end{split}
	\end{align}
	By the argument used to obtain \eqref{e08}, we conclude that $p(\vp) < q(\vp)$ for all $
	\vp \in (\vp_1,\vp_{j+1})$. Taking the limit as $\vp \rar \vp_1$ and translating back to the old variables
	yield
	\begin{align}\label{e16}
	\vp'(a) \leq \vp'(c) = 0, 
	\end{align}
	which contradicts \eqref{e14}.  Thus, $\vp_{j+2} \geq \vp_1$.  This concludes the proof for this case.  By induction, 
	we have proved \eqref{e12b} and also \eqref{e11}.  
	
	By definition, $\vp'(x) > 0$ on $(-\infty,x_1)$ and satisfies $\vp'(-\infty) = 0$.  By Lemma 
	\ref{c07b}, we must have $\vp'(x) < 0$ on $(x_N, x_{N-1})$.  In particular, by \eqref{e11}, we must have 
	$\vp_N = \vp(x_{N}) > \vp_1$.  Let $a < x_N$ be so that $\vp(a) = \vp_1$, and denote $b = x_N$.  Let 
	$d > x_1$ be so that $\vp(d) = \vp_N$, and denote $c = x_1$.  Then by definition, we have that    
	\begin{align}
	\begin{split}\label{e17}
	&\forall x \in (a,b) \cup (c,d), \quad \vp'(x) > 0,  \\
	&0 = \vp'(b) < \vp'(d), \\
	&a< b < c < d. 
	\end{split}
	\end{align}
	We can repeat the proof of Case 2 in the previous induction argument to conclude that 
	\begin{align*}
	\vp'(a) = \vp'(c) = 0,
	\end{align*}
	a contradiction to the fact that $a < x_N$ and $x_N$ is the least critical point.  Thus, the set of critical points must by 
	empty. 
	
\end{proof}

An immediate corollary of Lemma \ref{l07} and Lemma \ref{l06} is the uniqueness of the stationary Adkins--Nappi map of degree $n$. 

\begin{cor}\label{l08}
	The solution obtained in Proposition \ref{p01}, which we denote by $\gs_n$, is the unique solution to \eqref{e00}.  
\end{cor}


\section{Strichartz Estimates}


The main goal in this section is to establish Strichartz estimates for the linear inhomogeneous wave equation obtained from~\eqref{eq:an} by linearization about the stationary Adkins-Nappi map $Q = Q_n$. These linear estimates will be the foundation of our proof in Section~\ref{s:SD} that each $Q_n$ is unconditionally  asymptotically stable in the critical space $\HH_n$. 

We first make use of the additional dispersion  present in the equivariant setting by reformulating~\eqref{eq:an} as a semilinear wave equation in $\R^{1+5}$.

\subsection{$5d$ reduction: A reformulation of Theorem~\ref{t:main}}

It is convenient to write solutions $\vec \psi(t) \in \E_n$ to~\eqref{eq:an} as  perturbations of $Q = Q_n$ and then pass to a $5d$ re-formulation of~\eqref{eq:an} via the ansatz $\psi = Q + r u$.  This reduction is motivated by the presence
of a strong repulsive term in the potential which arises when linearizing about $Q$: 
\begin{align*}
\frac{2 \cos 2Q}{r^2} + \frac{1 - 2 \cos 2Q + 2Q \sin 2Q + \cos 4Q}{r^4} = \frac{2}{r^2} + O( \ang r ^{-6}). 
\end{align*}
The above asymptotics follow in a simple fashion from Theorem \ref{t01}.  The strong repulsive term $\frac{2}{r^2}$ indicates that the linearized operator about $Q$ has the same dispersion as the free
wave equation on $\R^{1+5}$.  Indeed, with $u(t,r)$ defined by $\psi(t,r) = Q(r) + r u(t,r)$ we see that $u$ satisfies
the $5d$ semilinear wave equation
\begin{align}
\begin{split}\label{e31}
&u_{tt}- u_{rr} - \frac{4}{r} u_r + V(r) u + Z(r,u) = 0, \\
&\vec u(0) = (u_0,u_1) \in \HH,
\end{split}
\end{align}
where
\EQ{ \label{eq:HHdef} 
	\HH := (\dot H^2 \times \dot H^1) \cap (\dot H^1 \times L^2)(\R^5). 
}
The potential $V$ is given by 
\begin{align}\label{e32}
V(r) = \frac{2(\cos 2Q - 1)}{r^2} + \frac{1 - 2 \cos 2Q + 2Q \sin 2Q + \cos 4Q}{r^4}, 
\end{align}
and the nonlinearity is given by 
\begin{align}\label{e312b}
Z(r,u) = u^2 Z_2(r) + u^3 Z_3(r,ru) + u^4 Z_4(r,ru) + u^5 Z_5(ru)
\end{align}
where 
\begin{align}
Z_2(r) &= r^{-1} \sin 2Q + 2 r^{-3} \bigl ( 9 \sin 2Q + Q \cos 2Q - 3Q - 4 \sin 4Q \bigr ), \label{e33}
\end{align}
\begin{align}
\begin{split}\label{e34}
Z_3(r,\vp) &= \cos 2Q \left ( \frac{\sin 2\vp - 2\vp}{\vp^3} \right )
+ \frac{\cos 2Q - 1}{2r^2} \left ( \frac{1 - \cos 2 \vp}{\vp^2} \right ) \\
&\:\: + \frac{2 Q \sin 2Q + 1 - \cos 2Q}{2r^2} \left ( \frac{\sin 2\vp - 2\vp}{\vp^3} \right ) \\
&\:\: + \frac{\cos 4Q - 1}{4r^2} \left ( \frac{\sin 4\vp - 4\vp}{\vp^3} \right ),
\end{split}
\end{align}
\begin{align}
\begin{split}\label{e35}
Z_4(r,\vp) &= r \sin 2Q \left ( \frac{\cos 2\vp - 1 + 2\vp^2}{\vp^4} \right ) \\
&\:\: + \frac{\sin 4Q - 2 \sin 2Q}{4r} \left ( \frac{\cos 4\vp - 1 + 8 \vp^2}{\vp^4} \right ) \\
&\:\: + \frac{Q ( \cos 2 Q - 1)}{r} \left ( \frac{1 - \cos 2 \vp - 2 \vp^2}{\vp^4} \right ) \\
&\:\: + \frac{\sin 2Q - 2Q}{r} \left ( \frac{\sin 2 \vp - 2 \vp}{\vp^3} \right ) \\
&\:\: + \frac{Q}{r} \left ( \frac{1 - \cos 2 \vp + 2 \vp \sin 2\vp + \cos 4 \vp}{\vp^4} \right ),
\end{split}
\end{align}
\begin{align}
Z_5(\vp) &= \frac{\vp - \frac{\sin 2 \vp}{2} - \vp \cos 2 \vp + \frac{\sin 4 \vp}{4}}{\vp^5}. \label{e36}
\end{align}
Using Theorem \ref{t01}, we see that we have the following elementary estimates on the potential and nonlinearity: 
\begin{align}\label{pot bound}
|V(r)| \lesssim \LR{r}^{-6}, 
\end{align}
as well as 
\begin{align}
\begin{split}\label{Z bounds}
\bigl | u^2 Z_2(r) \bigr | &\lesssim \LR{r}^{-3} |u|^2, \\
\bigl | u^3 Z_3(r,ru) \bigr | &\lesssim |u|^3, \\
\bigl | u^4 Z_4(r,ru) \bigr | &\lesssim \LR{r}^{-1} |u|^4, \\
\bigl | u^5 Z_5(ru) \bigr | &\lesssim |u|^5.
\end{split}
\end{align}

We claim that studying solutions $\vec \psi(t) \in \HH_n$ to \eqref{eq:an} is equivalent to studying the Cauchy problem \eqref{e31} with $\vec u(t) \in \HH$ where the relationship between the two is given by $\psi = Q + r u$.  Indeed, if $\vec \psi \in \HH_n$, we define $\vphi(t,r) = \psi(t,r) - Q(r)$ so that $\vphi(t,r) = r u(t,r)$. Then using Hardy's inequality and the relations
\begin{align*}
\vp_r &= ru_r + u, \\
\vp_{rr} &= r u_{rr} + 2 u_r, 
\end{align*}
once can show that for $s = 1,2$, the map 
\begin{align*}
\dot H^{s}_{rad}(\R^5) \ni u \mapsto \vphi := ru \in \dot H^s_{rad}(\R^3)
\end{align*}
is an isomorphism.  Thus, we conclude that 
\begin{align*}
\| \vec \psi \|_{\HH_n} = \| \vec \varphi \|_{(\dot H^2 \times \dot H^1)\cap (\dot H^1 \times L^2)(\R^3)} \simeq \| \vec u \|_{(\dot H^2 \times \dot H^1)
	\cap (\dot H^1 \times L^2)(\R^5)}
\end{align*}
which shows the two Cauchy problems are equivalent.

In the remainder of the paper we will work exclusively in the $5d$ 
$u$--formulation. We will first reformulate our main result, Theorem~\ref{t:main}(b).

\begin{thm}\label{t31}
	Let $(u_0,u_1) \in \cl H = (\dot H^2 \times \dot H^1) \cap (\dot H^1 \times \dot L^2)(\R^5)$.  Then there exists a unique 
	solution $\vec u(t) \in C(I_{\max}; \cl H)$ to \eqref{e31} with initial data $\vec u(0) = (u_0,u_1)$, defined on a 
	maximal interval of existence $I_{\max} = (T_-(u), T_+(u) \ni 0$.  If in addition, we assume that 
	\begin{align} \label{eq:cond} 
	\sup_{t \in [0, T_+(u))} \| \vec u(t) \|_{\cl H} < \infty, 
	\end{align}
	then $T_+ = \infty$.  Moreover, $\vec u(t)$ scatters to a free wave in $\cl H$ as $t \rar \infty$, i.e. there exists a radial solution 
	$\vec v(t) \in \cl H$ to the free wave equation on $\R^{1+5}$
	\begin{align}\label{e37}
	v_{tt} - v_{rr} - \frac{4}{r} v_r = 0, 
	\end{align}
	such that 
	\begin{align*}
	\| \vec u(t) - \vec v(t) \|_{\cl H} \rar 0 \mbox{ as } t \rar \infty. 
	\end{align*}
	A similar statement holds for negative times as well. 
\end{thm}

\subsection{Strichartz Estimates} 

The key element in our proof of the asymptotic stability of the $Q_n$ in Section~\ref{s:SD} are Strichartz estimates for the underlying linear wave equation. From~\eqref{e31} we see  that the relevant inhomogeneous linear equation in this paper is the following $(1+5)$-dimensional wave equation,  
\EQ{ \label{eq:lw} 
	&w_{tt} - \De w + V w = F,  \\
	&\vec w(0) = (w_0, w_1).
}
for radial initial data $\vec w(0)$, radial $F$, and where $V = V(r)$ is given by the explicit formula, 
\EQ{ \label{Vdef}
	V(r) = \frac{2(\cos 2Q - 1)}{r^2} + \frac{1 - 2 \cos 2Q + 2Q \sin 2Q + \cos 4Q}{r^4}. 
}
and satisfies the decay estimates
\EQ{\label{Vdec} 
	\abs{V(r)} \lesssim \ang{r}^{-6}
}
The goal is to show that despite the presence of the potential $V$,  the same family of Strichartz estimates that hold for the linear wave equation in $\R^{1+5}$, i.e., $\Box_{\R^{1+5}} w = F$, hold for ~\eqref{eq:lw} as well. Using a standard argument, the problem can be reduced to establishing localized energy estimates for the free evolution operator $S_V(t)$ generated by the left-hand side of~\eqref{eq:lw}, which in turn hinge on proving spectral properties of the Schr\"odinger operator
\EQ{ \label{eq:HVdef} 
	H_V:= - \De_{\R^5} + V.
}
Recall that a triple $(p, q, \gamma)$ is called \emph{admissible} for the wave equation in $\R^{1+5}$  if 
\EQ{
	p>2, q \ge 2, \quad \frac{1}{p} + \frac{5}{q} = \frac{5}{2} - \gamma, \quad \frac{1}{p} + \frac{1}{ q} \le 1
} 
We prove the following results. 
\begin{thm}\label{t:lin} The following hold true. 
	
	\begin{itemize}
		\item[(a)]  \label{t:Hspec} The operator $H_V$ defined in~\eqref{eq:HVdef} for $V = V(r)$ as in~\eqref{Vdef},~\eqref{Vdec} is self-adjoint on $\calD:= H^2(\R^5)$. Its spectrum is given by $[0, \infty)$ and is purely absolutely continuous.  In particular, $H_V$ has no negative spectrum and the threshold energy $0$ is neither an eigenvalue or a resonance.
		\item[(b)] \emph{[Strichartz Estimates]}\label{t:strich} Let $\vec w(t)$ be a solution to~\eqref{eq:lw} defined on a time interval $I \ni 0$ with radial intial data $\vec  w(0) = (w_0, w_1)$. Then for any admissible triples $(p, q, \gamma)$ and $(a, b, \sigma)$ we have 
		\EQ{
			\| \abs{\na}^{- \gamma}  \na_{t, x} w \|_{L^p_t, L^q_x(I)} 
			\lesssim \| (w_0, w_1) \|_{ \dot{H}^{1} \times L^2(\R^5)} + \|  \abs{\na}^\sigma F\|_{L^{a'}_t L^{b'}_x(I)}
		}
	\end{itemize} 
\end{thm} 

A general (and by now standard) argument that we outline in Appendix~\ref{a:S} shows that Theorem~\ref{t:strich}(b) holds for any smooth bounded  potential $V(r)$ decaying like~\eqref{Vdec} and corresponding Schr\"odinger operator $H_V$  satisfying the conclusions of Theorem~\ref{t:strich}(a). Here we prove part (a) for the operator $H_V$.  

\begin{proof}[Proof of Theorem~\ref{t:strich} \emph{(a)}]
	It is convenient to conjugate to a half-line operator via the $L^2$ isometry 
	\EQ{
		L^2_{\rad} ( \R^5) \ni f \mapsto r^2 f=  \phi \in L^2((0, \infty), dr) 
	}
	We then define $\LL_V$ by 
	\EQ{
		H_V f = r^{-2} \LL_V (r^2 f) 
	}
	which yields, 
	\EQ{
		\LL_V = -\p_r^2 + \frac{2}{r^2} + V(r).
	}
	From the decay properties of $V$ we see that it suffices to show that $\LL_V$ has no negative spectrum and that $0$ is neither an eigenvalue or a resonance. In fact, we can quickly rule out a resonance at $0$. Indeed $\{ r^2, r^{-1}\}$ is a fundamental system for $ \LL_0 := -\p_r^2 + \frac{2}{r^2}$. Since $V(r)  \lesssim \ang{r}^{-6}$,  any bounded solution $\phi_0$ to $\LL_V \phi= 0$ satisfies  $ \abs{\phi_0(r)}  \lesssim r^{-1}$ as $r \to \infty$  and  $ \abs{\phi_0(r)} \lesssim r^2$ as $r \to 0$. This means that $\phi_0 \in L^2(0, \infty)$ and is thus an eigenvalue.  
	
	We now show that $\LL_V$ has no negative spectrum and no eigenvalue at $0$. The proof will follow by a Sturm comparison argument where we find a positive threshold eigenvector for a half-line operator with a deeper potential well. To motivate our guess for this eigenvector and operator we note  that the last term on the right-hand-side of~\eqref{eq:an}
	is always nonnegative and is thus de-focusing in nature and that this term is the only difference between~\eqref{e00} and the harmonic maps equation in $3d$.  
	Next, recall that the harmonic maps equation in $\R^3$ is invariant under the  $\dot{H}^{\frac{3}{2}}(\R^3)$-invariant scaling given by 
	\EQ{
		g \mapsto g_{\la} = g( r/ \la)
	}
	and with infinitesimal generator 
	\EQ{
		\La g = - \frac{d}{d \la} g_{\la} \rest_{\la = 1}  =  (x \cdot \na )g = r \p_r g
	}
	Consider $\La Q = r \p_r Q(r)$. Passing to the half-line via $\La Q \mapsto r \La Q$ we obtain 
	\EQ{
		\Phi(r):=  r \La Q = r^2 Q_r(r)
	}
	which solves 
	\EQ{\label{eq:tiV}
		(\LL_V - \ti V) \Phi = 0
	}
	where 
	\EQ{
		\ti V(r)  =  2 \frac{( Q - \sin Q \cos Q)(1- \cos 2 Q)}{ r^5 Q_r}
	}
	To prove~\eqref{eq:tiV} simply multiply~\eqref{e00} by $r^2$ and differentiate in $r$. 
	
	By Theorem~\ref{t01} we see that  $|\ti V(r)| \lesssim 1$ and thus $\calD( \LL_V - \ti V)  = \calD( \LL_V)$, where $\calD( \LL)$ denotes the domain of $\LL$.    In fact, since $0< Q(r) < n \pi$ and $Q_r(r) > 0$ on $(0, \infty)$ we have $\ti V(r) > 0$ for all $r \in (0, \infty)$ except at only finitely many points (where $Q(r)$ equals an integer or half--integer multiple of $\pi$).  Moreover, since $Q_{r}(r) \lesssim  r^{-3}$ as $r \to \infty$ it follows that $\Phi(r)$ is a eigenvalue at $0$ for $\LL_V - \ti V$. Since $Q_r >0$ for all $r \in (0, \infty)$ it's immediate from Sturm oscillation theory that $\LL_V - \ti V$ has no negative spectrum.  In particular, by a variational principle we must have 
	\EQ{
		\langle f  \mid (\LL_V - \ti V) f\rangle_{L^2(0, \infty) }  \ge 0 \quad \forall f \in H^2(0, \infty)
	}
	and therefore 
	\EQ{
		\langle f  \mid \LL_V  f\rangle_{L^2(0, \infty) }  \ge  \int_0^\infty \ti V(r) |f(r)|^2 \, dr  \quad \forall f \in H^2(0, \infty).
	}
	Since $\ti V(r) >0$ on $(0,\infty)$, modulo a finite set,  the above precludes $\LL_V$ having negative spectrum or an eigenvalue at $0$. We remark that the non--negativity of the last term in~\eqref{eq:an} along with the fact that $Q(r)$ is strictly increasing are what led to the crucial non--negativity of $\ti V(r)$. This completes the proof of Theorem~\ref{t:strich}(a). 
\end{proof} 

See Appendix~\ref{a:S} for a brief outline of the proof of Theorem~\ref{t:strich}(b). 


\section{Asymptotic Stability of $Q_n$}

\label{s:SD} 
In this section, we establish the theory for solutions to~\eqref{eq:an} which are small perturbations of the stationary Adkins--Nappi map $Q_n$.  In particular, we prove that $\vec Q_n$ is  asymptotic stable under perturbations that are small in  $\HH_0$.

We remark that the appearance  of the  $\dot{H}^2 \times \dot{H^1}$ in the definition of the norms $\HH$, $\HH_n$ and $\HH_0$ is due to the fact that the top order nonlinearity in the equation~\eqref{e31} is quintic and  $\dot{H}^2 \times \dot{H^1}$  is the scale invariant norm for the underlying scale invariant quintic nonlinear wave equation in $\R^{1+5}$. 
\subsection{Small data and perturbative theory}

We now formulate the local well--posedness and small data global well--posedness theory for \eqref{e31}.  The crucial ingredients
for establishing these theories are the Strichartz estimates established in Section 3 for the perturbed radial $5d$ wave equation
\begin{align}
\begin{split}\label{e38}
&v_{tt} - v_{rr} - \frac{4}{r} v_r  + V(r) v = F, \\
&\vec v(0) = (v_0,v_1), 
\end{split}
\end{align}
where $V$ is as in \eqref{e32}. In particular, for $I$ a time interval, if we denote 
\begin{align*}
S(I) &= L^3_t W^{1, \frac{30}{7}}_x(I) \cap L^5_t W^{1, \frac{50}{13}}_x(I), \\
W(I) &= L^{3}_t \bigl ( \dot W^{\frac{1}{2}, 3}_x \cap \dot W^{\frac{3}{2}, 3}_x \bigr )(I), \\
N(I) &= L^{\frac{3}{2}}_t \bigl ( \dot W^{\frac{1}{2}, \frac{3}{2}}_x \cap 
\dot W^{\frac{3}{2}, \frac{3}{2}}_x \bigr )(I) + L^1_t H^1_x(I),
\end{align*}
then a solution $v$ to \eqref{e38} on $I \ni 0$ satisfies the Strichartz estimate 
\begin{align}\label{e39}
\| \vec v \|_{L^\infty_t \cl H(I)} + \| v \|_{W(I)} + \| v \|_{S(I)} \lesssim \| \vec v(0) \|_{\cl H} 
+ \| F \|_{N(I)}.
\end{align}

Other essential tools that will be used in the proofs are the following radial Hardy--Sobolev embedding and product rule for 
fractional derivatives. 

\begin{lem}[Lemma 3, \cite{CKM}] \label{l31} 
	Let $d \geq 3$, and $p,q,\al,\beta \geq 0$ be such that $1 \leq q \leq p \leq \infty$, and 
	$0 < (\beta - \al) q < d$.  Then there exists $C = C(d,p,q,\al,\beta)$ such that for all $v$ radial in 
	$\R^d$, 
	\begin{align*}
	\left \| r^{\frac{d}{q} - \frac{d}{p} - \beta + \al} v \right \|_{\dot W^{\al, p}} \leq C \| v \|_{\dot W^{\beta, q}}.  
	\end{align*}
\end{lem}

\begin{lem}[Lemma 5, \cite{CKM}]\label{l32}
	Let $1 < p < \infty$, $0 < \al < 1$.  Then 
	\begin{align*}
	\| v w \|_{\dot W^{\al, p}} \lesssim \| v \|_{L^{p_1}} \| w \|_{\dot W^{\al, p_2}} +
	\| v \|_{\dot W^{\al, p_3}} \| w \|_{L^{p_4}},
	\end{align*}
	where $\frac{1}{p} = \frac{1}{p_1} + \frac{1}{p_2} = \frac{1}{p_3} + \frac{1}{p_4}$ and $1 < p_2, p_3 < \infty$.  
\end{lem}

\begin{prop}\label{p32}
	Let $(u_0,u_1) \in \cl H$.  Then there exists a unique solution $C(I_{\max}; \cl H)$ to \eqref{e31} with $\vec u(0) = (u_0,u_1)$ 
	defined on a maximal time interval of existence $I_{\max}(u) = (T_-(u), T_+(u))$ such that for any $J \Subset I_{\max}$, 
	\begin{align*}
	\| \vec u \|_{L^\infty_t \cl H(J)} + \| u \|_{W(J)} + \| u \|_{S(J)} < \infty.  
	\end{align*}
	A solution $\vec u(t)$ with $T_+(u) = \infty$ scatters to a free wave, i.e. a solution to \eqref{e37}, if and only if 
	$\| u \|_{S(0,\infty)} < \infty$.  A similar statement holds for negative times.  In particular, there exists $\delta > 0$ 
	such that if $\| \vec u(0) \|_{\cl H} < \delta$, then $I_{\max} = \R$ and 
	\begin{align}\label{e310}
	\| \vec u \|_{L^\infty_t \cl H(\R)} + \| u \|_{W(\R)} + \| u \|_{S(\R)} \lesssim \| \vec u(0) \|_{\cl H}, 
	\end{align}
	so that $\vec u(t)$ scatters to free waves as $t \rar \pm \infty$.  Finally, we have the standard finite time blow--up criterion
	that $T_+(u) < \infty$ if and only if $\| u \|_{S(0,T_+)} = \infty$ with a similar statement holding for negative times. 
\end{prop}

\begin{rem}
	We note that Proposition \ref{p32} establishes the unconditional asymptotic stability of the stationary Adkins--Nappi map $Q_n$, i.e., Theorem~\ref{t:main}(a).  
\end{rem} 

\begin{proof}
	The proofs of local well--posedness and the blow--up criterion follow from the standard Duhamel formulation, contraction mapping and continuity arguments using the Strichartz estimate \eqref{e39}.  
	For completeness, we prove the a priori small data global estimate \eqref{e310} which gives the overall flavor of the arguments.  
	
	Let $I \subset \R$, $0 \in I$, and let
	$u$ be a solution to \eqref{e31}.  By \eqref{e39}, we have that
	\begin{align}\label{e311}
	\| \vec u \|_{L^\infty_t \cl H(I)} + \| u \|_{W(I)} + \| u \|_{S(I)} \lesssim \| \vec u(0) \|_{\cl H} 
	+ \| Z(\cdot, u) \|_{N(I)}.
	\end{align}
	We estimate the right--hand side of \eqref{e311} via 
	\begin{align}
	\begin{split}\label{e312}
	\| Z(\cdot, u) \|_{N(I)} 
	&\leq \| u^2 Z_2(r) \|_{ L^{\frac{3}{2}}_t \bigl ( \dot W^{\frac{1}{2}, \frac{3}{2}}_x \cap 
		\dot W^{\frac{3}{2}, \frac{3}{2}}_x \bigr )(I)} + \| u^3 Z_3(r,ru) \|_{L^1_t H^1_x(I)} \\
	&\:\:+ \| u^4 Z_4(r,ru) \|_{L^1_t H^1_x(I)} + \| u^5 Z_5(ru) \|_{L^1_t H^1_x(I)}.  
	\end{split}
	\end{align}
	
	We first estimate the quadratic term in \eqref{e312}.  Note that by \eqref{e33} and 
	Theorem \ref{t01}, $|Z_2^{(j)}(r)| \lesssim \ang r^{-3 - j}$.  By Lemma \ref{l32}, 
	\begin{align*}
	\| u^2 Z_2(r) \|_{\dot W^{\frac{1}{2}, \frac{3}{2}}}
	&\lesssim \| u^2 \|_{L^{\frac{15}{7}}} \| Z_2 \|_{\dot W^{\frac{1}{2}, 5}} + 
	\| u^2 \|_{\dot W^{\frac{1}{2}, \frac{30}{17}}} \| Z_2 \|_{L^{10}} \\
	&\lesssim \| u \|_{L^{\frac{30}{7}}}^2 + \| u \|_{\dot W^{\frac{1}{2}, 3}} \| u \|_{L^{\frac{30}{7}}}.  
	\end{align*}
	Similarly 
	\begin{align*}
	\| u^2 Z_2(r) \|_{\dot W^{\frac{3}{2}, \frac{3}{2}}} 
	&\simeq \| \p_r (u^2 Z_2(r)) \|_{\dot W^{\frac{1}{2}, \frac{3}{2}}} \\
	&\lesssim \| u_r u Z_2(r) \|_{{\dot W^{\frac{1}{2}, \frac{3}{2}}}} +
	\| u^2 Z_2'(r) \|_{{\dot W^{\frac{1}{2}, \frac{3}{2}}}} \\
	&\lesssim \| u_r \|_{L^{\frac{30}{7}}} \| u \|_{L^{\frac{30}{7}}} + \| u_r \|_{\dot W^{\frac{1}{2}, 3}} \| u \|_{L^{\frac{30}{7}}} \\
	&\:\:+ \| u_r \|_{L^{\frac{30}{7}}} \| u \|_{\dot W^{\frac{1}{2}, 3}} + \| u \|^2_{L^{\frac{30}{7}}} \\
	&\:\:+ \| u \|_{\dot W^{\frac{1}{2}, 3}}\| u \|_{L^{\frac{30}{7}}}.  
	\end{align*}
	Thus, 
	\begin{align*}
	\| u^2 Z_2(r) \|_{L^{\frac{3}{2}}_t \bigl ( \dot W^{\frac{1}{2}, \frac{3}{2}}_x \cap 
		\dot W^{\frac{3}{2}, \frac{3}{2}}_x \bigr )(I)}
	\lesssim \| u \|^2_{W^{1, \frac{30}{7}}} + \| u \|_{\dot W^{\frac{1}{2}, 3}_x \cap \dot W^{\frac{3}{2}, 3}} 
	\| u \|_{W^{1, \frac{30}{7}}},
	\end{align*}
	whence by H\"older's inequality in time we see that 
	\begin{align}\label{e313}
	\| u^2 Z_2(r) \|_{L^{\frac{3}{2}}_t \bigl ( \dot W^{\frac{1}{2}, \frac{3}{2}}_x \cap 
		\dot W^{\frac{3}{2}, \frac{3}{2}}_x \bigr )(I)} 
	\lesssim \| u \|^2_{S(I)} + \| u \|_{W(I)} \| u \|_{S(I)}. 
	\end{align}
	
	We will now estimate the cubic term.  We note that $u^3 Z_3(r,ru)$ is a sum of terms of the form $z(r) u^3 Z(ru)$ where
	for all $j \geq 0$ 
	\begin{align}
	\begin{split}\label{e314}
	|z^{(j)}(r)| &\lesssim 1, \\
	|Z^{(j)}(\vp)| &\lesssim 1. 
	\end{split}
	\end{align}
	Thus, to estimate $u^3 Z(r,ru)$ in $L^1_t H^1_x(I)$, it suffices to estimate a nonlinearity of the form $z(r)u^3 Z(ru)$ where $z,Z$ satisfy 
	\eqref{e314}.  We now estimate
	\begin{align*}
	\| z(r) u^3 Z(ru) \|_{L^2} 
	&\lesssim \| u^3 \|_{L^2} \\
	&\lesssim \| u \|_{L^{\frac{30}{7}}} \| u \|^2_{L^{\frac{15}{2}}}. 
	\end{align*}
	By Sobolev embedding we have  
	\begin{align}\label{e315a}
	W^{1, \frac{30}{7}} \hookrightarrow L^p, \quad \forall p \in \left [ \frac{30}{7}, 30 \right ] 
	\end{align}
	which along with the previous estimate implies that
	\begin{align*}
	\| z(r) u^3 Z(ru) \|_{L^2} \lesssim \| u \|^3_{W^{1, \frac{30}{7}}}.  
	\end{align*}
	Next, we have 
	\begin{align*}
	\| z(r) u^3 Z(ru) \|_{\dot H^1} &= \| \p_r (z(r) u^3 Z(ru)) \|_{L^2} \\
	&\lesssim \| z'(r) u^3 Z(ru) \|_{L^2} + 
	\| z(r) u_r u^2 Z(ru)) \|_{L^2} \\ &\:\:+ 
	\| z(r) u^3 Z'(ru) (ru)_r \|_{L^2}. 
	\end{align*}
	The first two terms are estimated as in the previous $L^2$ estimate and we obtain 
	\begin{align*}
	\| z'(r)& u^3 Z(ru) \|_{L^2} + 
	\| z(r) u_r u^2 Z(ru)) \|_{L^2} \\
	&\lesssim \| u^3 \|_{L^2} + \| u_r u^2 \|_{L^2} \\
	&\lesssim \bigl [ \| u \|_{L^{\frac{30}{7}}} + \| u_r \|_{L^{\frac{30}{7}}} \bigr ] \| u \|^2_{L^{\frac{15}{2}}} \\
	&\lesssim \| u \|^3_{W^{1, \frac{30}{7}}}.
	\end{align*}
	The remaining term is estimated using \eqref{e314}, Lemma \ref{l32} and \eqref{e315a}:
	\begin{align*}
	\| z'(r) u^3 Z'(ru) (ru)_r \|_{L^2} 
	&\lesssim \| u^3 (ru)_r \|_{L^2} \\
	&\lesssim \| u \|_{L^{\frac{15}{2}}}^3 \| (ru)_r \|_{L^{10}} \\
	&\lesssim \| u \|^3_{W^{1, \frac{30}{7}}} \| u \|_{\dot H^2}.  
	\end{align*}
	Thus, 
	\begin{align*}
	\| z(r) u^3 Z(ru) \|_{H^1} \lesssim \| u \|^3_{W^{1, \frac{30}{7}}} + \| u \|^3_{W^{1, \frac{30}{7}}} \| u \|_{\dot H^2}.  
	\end{align*}
	As noted previously, $u^3 Z_3$ is a sum of nonlinearities of the previous form estimated whence by integration
	in time we obtain 
	\begin{align}\label{e316}
	\| u^3 Z_3(r, ru) \|_{L^1_t H^1_x(I)} \lesssim \| u \|_{S(I)}^3 + \| \vec u \|_{L^\infty_t \cl H(I)} \| u \|_{S(I)}^3.
	\end{align}
	
	We will now turn to the quartic term in \eqref{e312}.  Similar to the previous estimate for the cubic term, it suffices to estimate
	a nonlinearity of the form $z(r) u^4 Z(ru)$ in $L^1_t H^1_x$ with $z, Z$ satisfying \eqref{e314}.  Before 
	doing so, we first state some useful embeddings. By Sobolev embedding, we have 
	\begin{align}\label{e317}
	W^{1, \frac{50}{13}}_x \hookrightarrow L^q_x, \quad \forall q \in \left [ \frac{50}{13}, \frac{50}{3} \right ].  
	\end{align}
	Then by \eqref{e315a} with $p = 15$, \eqref{e317} with $q = \frac{25}{3}$ and interpolation we see that 
	\begin{align*}
	S(I) \hookrightarrow L^4_t L^{10}_x(I). 
	\end{align*}
	Similarly, by \eqref{e315a} with $p = 30$, \eqref{e317} with $q = \frac{50}{3}$ and interpolation we see that 
	\begin{align*}
	S(I) \hookrightarrow L^4_t L^{20}_x(I).
	\end{align*}
	Thus, we have the embedding
	\begin{align}\label{e317b}
	S(I) \hookrightarrow  L^4_t L^{10}_x(I) \cap  L^4_t L^{20}_x(I).
	\end{align}
	We now estimate $u^4 Z_4$.  By 
	\eqref{e314} we have
	\begin{align*}
	\| z(r) u^4 Z(ru) \|_{L^2} &\lesssim \| u^4 \|_{L^2} \\
	&\lesssim \| u \|_{L^{\frac{10}{3}}} \| u \|^3_{L^{15}} \\
	&\lesssim \| u \|_{\dot H^1} \| u \|^3_{W^{1,\frac{30}{7}}},
	\end{align*}
	where we used the Sobolev embedding, \eqref{e315a}, to pass to the final line.  Next, we have 
	\begin{align}
	\begin{split}\label{e317a}
	\| z(r) u^4 Z(ru) \|_{\dot H^1} &= \| \p_r (z(r) u^4 Z(ru)) \|_{L^2} \\
	&\lesssim \| z'(r) u^4 Z(ru) \|_{L^2} + 
	\| z(r) u_r u^3 Z(ru)) \|_{L^2} \\ &\:\:+ 
	\| z(r) u^4 Z'(ru) (ru)_r \|_{L^2}. 
	\end{split}
	\end{align}
	The first two terms above are estimated as in the previous $L^2$ estimate and we obtain 
	\begin{align*}
	\| z'(r) u^4 Z(ru) \|_{L^2} + 
	\| z(r) u_r u^3 Z(ru)) \|_{L^2} &\lesssim [ \| u \|_{L^{\frac{10}{3}}} + \| u_r \|_{L^{\frac{10}{3}}} ] \| u \|^3_{L^{15}} \\
	&\lesssim \| u \|_{\dot H^2 \cap \dot H^1} \| u \|^3_{W^{1,\frac{30}{7}}}.
	\end{align*}
	We use the radial Sobolev embedding on the third term in \eqref{e317a} and obtain 
	\begin{align*}
	\| z(r) u^4 Z'(ru) (ru)_r \|_{L^2} &\lesssim \| u \|^4_{L^{10}} \| (ru)_r \|_{L^{10}} \\
	&\lesssim \| u \|^4_{L^{10}} \| u \|_{\dot H^2}.  
	\end{align*}
	Thus, 
	\begin{align*}
	\| z(r) u^4 Z(ru) \|_{H^1_x} \lesssim \| u \|_{\dot H^1 \cap \dot H^2} \| u \|^3_{W^{1,\frac{30}{7}}} + \| u \|_{\dot H^2}
	\| u \|^4_{L^{10}}.  
	\end{align*}
	As noted previously, $u^4 Z_4$ is a sum of nonlinearities of the previous form estimated whence by integration
	in time and \eqref{e317b} we obtain 
	\begin{align}\label{e319}
	\| u^4 Z_4(r, ru) \|_{L^1_t H^1_x(I)} \lesssim  \| \vec u \|_{L^\infty_t \HH (I)} \| u \|_{S(I)}^3 + 
	\| \vec u \|_{L^\infty_t \cl H(I)} \| u \|_{S(I)}^4.
	\end{align} 
	
	Finally, we estimate the quintic term in \eqref{e312}.  We first note that $|Z_5^{(j)}(\vp)| \lesssim 1$ for all $j \geq 0$.
	Then 
	\begin{align*}
	\| u^5 Z_5(ru) \|_{L^2} &\lesssim \| u^5 \|_{L^2} \\
	&\lesssim \| u \|_{L^{\frac{10}{3}}} \| u \|^4_{L^{20}} \\
	&\lesssim \| u \|_{\dot H^1} \| u \|^4_{L^{20}}. 
	\end{align*}
	Next we use radial Sobolev embedding to see that 
	\begin{align*}
	\| u^5 Z_5(ru) \|_{\dot H^1} &\lesssim \| u_r u^4 \|_{L^2} + \| u^5 (ru)_r \|_{L^2} \\
	&\lesssim \| u_r \|_{L^{\frac{10}{3}}} \| u \|^4_{L^{20}}
	+  \| u \|^5_{L^{\frac{25}{2}}} \| (ru)_r \|_{L^{10}} \\
	&\lesssim \| u \|_{\dot H^2} \| u \|^4_{L^{20}} + \| u \|_{\dot H^2} \| u \|^5_{W^{1, \frac{50}{13}}}.
	\end{align*}
	Thus, 
	\begin{align*}
	\| u^5 Z_5(ru) \|_{H^1} \lesssim 
	\| u \|_{\dot H^1 \cap \dot H^2} \| u \|^4_{L^{20}} + \| u \|_{\dot H^2} \| u \|^5_{W^{1, \frac{50}{13}}}.
	\end{align*}
	By integrating in time and \eqref{e317b} we conclude that 
	\begin{align}\label{e320}
	\| u^5 Z_5(ru) \|_{L^1_t H^1_x(I)} \lesssim  \| \vec u \|_{L^\infty_t \HH (I)} \| u \|_{S(I)}^4 + 
	\| \vec u \|_{L^\infty_t \cl H(I)} \| u \|_{S(I)}^5.
	\end{align}
	
	If we define $X(I) = L^\infty_t \HH(I) \cap S(I) \cap W(I)$ and combine 
	\eqref{e312}, \eqref{e313}, \eqref{e316}, \eqref{e319}, \eqref{e320} we obtain 
	\begin{align*}
	\| u \|_{X(I)} \lesssim& \| \vec u(0) \|_{\HH} + \| u \|_{X(I)}^2
	+ \| u \|_{X(I)}^3 \\&+\:\:\:\| u \|_{X(I)}^4 + \| u \|_{X(I)}^5 + \| u \|_{X(I)}^6
	\end{align*}
	which finishes the proof of \eqref{e310} via a standard continuity argument.
	
	A simple variant of the above argument also shows that 
	if $\| u \|_{S(0,\infty)} < \infty$, then 
	\begin{align*}
	\vec w_L(0) = \vec u(0) - \int_0^\infty \vec S_V(-s)(0,Z(\cdot, u(s))) ds 
	\end{align*}
	converges in $\HH$.  Thus, by Duhamel we conclude that 
	\begin{align}
	\vec u(t) = \vec S_V(t) \vec w_L(0) + o_{\HH}(1), \label{p scatt}
	\end{align}
	as $t \rar \infty$, that is, $\vec u(t)$ scatters to a solution to the perturbed $5d$ wave equation \eqref{e38} with $F = 0$.  To extract a solution $v_L(t) = S(t)\vec v_L(0)$ to the free wave equation \eqref{e37} from the perturbed wave $w_L(t) = S_V(t)\vec w_L(0)$,
	we write via Duhamel 
	\begin{align*}
	w_L(t) &= S(t) \vec w_L(0) + \int_0^t S(t-s)(0,Vw_L(s)) ds \\
	&= S(t)\left [ \vec w_L(0) + \int_0^t S(-s)(0, Vw_L(s)) ds \right ]. 
	\end{align*}
	We then take 
	\begin{align*}
	\vec v_L(0) = \vec w_L(0) + \int_0^\infty S(-s)(0, Vw_L(s)) ds
	\end{align*}
	which converges in $\HH$ (see Appendix A).  Then $\vec w_L(t) = \vec v_L(t) + o_{\HH}(1)$ as $t \rar \infty$.  This along with \eqref{p scatt}
	allow us to conclude that if $\| u \|_{S(0,\infty)} < \infty$, then $u$ scatters to a free wave in $\HH$ as 
	$t \rar \infty$.  The fact that the finiteness of $\| u \|_{S(0,\infty)}$ is necessary if $u$ scatters
	as $t \rar \infty$ follows from similar arguments using the fact that $\| v_L \|_{S(0,\infty)} < \infty$ holds
	for any free wave on $\R^{1+5}$. This concludes the proof. 
\end{proof}

\begin{prop}[Perturbation lemma]\label{p33}
	Let $A > 0$.  Then there exist $\e_0 = \e_0(A) > 0$ and $C = C(A) > 0$ such that the following holds.  
	Let $0 < \e < \e_0$, $(u_0,u_1) \in \HH$ and $I \subseteq \R$ with $0 \in I$.  Assume that $\vec v(t) \in C(I; \HH)$ satisfies 
	on $I$ (in the sense of the integral 
	equation) 
	\begin{align}\label{e321}
	\Box v + V(r) v + Z(\cdot, v) = e
	\end{align}
	and that for all $J \Subset I$, $\| v \|_{W(J)} < \infty$ as well as 
	\begin{align}
	\| \vec v(0) \|_{\HH} + \| v \|_{S(I)} \leq A, \label{e322} \\
	\| (u_0,u_1) - \vec v(0) \|_{\HH} + \| e \|_{N(I)} \leq \e. \label{e323}
	\end{align}
	Then there exists a unique solution $\vec u(t) \in C(I; \HH)$ to \eqref{e31} with $\vec u(0) = (u_0,u_1)$ such that 
	\begin{align}
	\| \vec u - \vec v \|_{L^\infty_t \HH(I)} + \| u - v \|_{S(I)} \leq C(A) \e. \label{e324}
	\end{align}
	In particular, $\| u \|_{S(I)} < \infty$.  
	
\end{prop}

\begin{proof}
	It suffices to derive a priori estimates
	for $u$, assuming that it exists.  The local well--posedness theory, Proposition \ref{p32}, gives the rest.  We will 
	also assume, without loss of generality, that $I = [0,T)$ for some $T \in (0,\infty]$.  We first note that 
	\eqref{e322} implies that 
	\begin{align*}
	\| \vec v \|_{L^\infty_t \HH(I)} + \| v \|_{W(I)} \leq C(A).  
	\end{align*}
	Indeed, let $\eta > 0$ be small, to be determined, and partition $I$ into subintervals $I = \cup_{j = 1}^{J(A)} I_j$ such that 
	for all $j$, $\| v \|_{S(I_j)} < \eta$.  By Duhamel, we have for $t \in I_j = [t_j, t_{j+1}]$, 
	\begin{align*}
	v(t) = S_V(t-t_j) \vec v(t_j) + \int_{t_j}^t S_V(t-s) \bigl ( 0, e(s) - Z(\cdot, v(s)) \bigr ) ds.
	\end{align*}
	In what follows, we will use the notation $X(J) = L^\infty_t \HH(J) \cap W(J)$.  Then by Strichartz estimates and by arguing as in the proof of Proposition 
	\ref{p32}, we have for every $1 \leq j \leq J(A)$
	\begin{align*}
	\| v \|_{X(I_j)} &\lesssim \| \vec v(t_j) \|_{\HH} + \| Z(\cdot, v) \|_{N(I_j)} + \| e \|_{N(I_j)} \\
	&\lesssim \| \vec v(t_j) \|_{\HH} + \| v \|_{X(I_j)} \Bigl [ \| v \|_{S(I_j)} + 
	\| v \|_{S(I_j)}^2 \\&\hspace{1cm}+ \| v \|_{S(I_j)}^3 + \| v \|_{S(I_j)}^4 +  \| v \|_{S(I_j)}^5 \Bigr ] + \e \\
	&\lesssim \| \vec v(t_j) \|_{\HH} + c(\eta) \| v \|_{X(I_j)} + \e, 
	\end{align*}
	where $c(\eta) \rar 0$ as $\eta \rar 0$.  If we fix $\eta$ sufficiently small, then 
	\begin{align*}
	\| v \|_{X(I_j)} \lesssim \| \vec v(t_j) \|_{\HH} + \e. 
	\end{align*}
	In particular, for all $j$, $\| \vec v(t_{j+1}) \|_{\HH} \lesssim \| \vec v(t_j) \|_{\HH} + \e$ whence
	for all $j$, $\| \vec v(t_j) \|_{\HH} \lesssim C(J) [ \| \vec v(0) \|_{\HH} + 1 ]$.  We conclude that 
	$\| v \|_{X(I)} \leq C_1(A)$ as desired.  
	
	We now turn to establishing \eqref{e323}.     
	Let $w = v - u$.  Then $w$ verifies on $I$
	\begin{align*}
	&\Box w + V(r) w = Z(\cdot, v + w) - Z(\cdot, v) + e, \\
	&\vec w(0) = \vec v(0) - (u_0,u_1).
	\end{align*}
	Let $\eta > 0$ be small, to be determined later, and partition $I$ into subintervals $I = \cup_{j = 1}^{J(A)} I_j$ such that 
	\begin{align*}
	\| v \|_{W(I_j)} + \| v \|_{S(I_j)} < \eta.
	\end{align*}
	By Duhamel, we have for $t \in I_j = [t_j, t_{j+1}]$, 
	\begin{align*}
	w(t) = S_V(t-t_j) \vec w(t_j) + \int_{t_j}^t S_V(t-s) \bigl ( 0, e + Z(\cdot, v + w)- Z(\cdot, v) \bigr ) ds.
	\end{align*}
	In what follows, we will use the notation $X(J) = L^\infty_t \HH(J) \cap W(J) \cap S(J)$.  Then
	by Strichartz estimates and by arguing as in the proof of 
	Proposition \ref{p32} we see that for every $1 \leq j \leq J(A)$ 
	\begin{align*}
	\| w \|_{X(I_j)} &\lesssim \| \vec w(t_j) \|_{\HH} + \| Z(\cdot, v+w) - Z(\cdot, v) \|_{N(I_j)} + \| e \|_{N(I_j)} \\
	&\leq C_0 \| \vec w(t_j) \|_{\HH} + c(\eta) C_1(A) \Bigl [ \| w \|_{X(I_j)} + \| w \|_{X(I_j)}^2 \\
	&\hspace{1cm} \| w \|_{X(I_j)}^3 + \| w \|_{X(I_j)}^4 + \| w \|_{X(I_j)}^5 + \| w \|_{X(I_j)}^6 \Bigr ]  + C_0 \e, 
	\end{align*}
	where $c(\eta) \rar 0$ as $\eta \rar 0$ and $C_0$ is an absolute constant.  The point here is that when estimating the 
	individual nonlinearities in \eqref{e312b} that comprise $\| Z(\cdot, v+w) - Z(\cdot, v) \|_{N(I_j)}$, we always 
	can obtain at least one power of $\| v \|_{W(I_j)} + \| v \|_{S(I_j)}$ which is smaller than $\eta$. Thus, if we choose 
	$\eta$ sufficiently small, we conclude that 
	\begin{align*}
	\| w \|_{X(I_j)}
	&\leq 2C_0 \| \vec w(t_j) \|_{\HH} + 2C_1(A) \Bigl [ \| w \|_{X(I_j)}^2 + \| w \|_{X(I_j)}^3 \\
	&\hspace{1cm} \| w \|_{X(I_j)}^4 + \| w \|_{X(I_j)}^5 + \| w \|_{X(I_j)}^6 \Bigr ]  + 2C_0 \e. 
	\end{align*}
	From the definition of $X(J)$ and the continuity method (i.e. the fact that $X(J)$ is continuous with 
	respect to the upper endpoint), there exists $\delta_0 = \delta_0(A) > 0$ such that if 
	$\| \vec w(t_j) \|_{\HH} + \e < \delta_0$, then 
	\begin{align*}
	\| w \|_{X(I_j)} \leq 4C_0 \bigl [ \| \vec w(t_j) \|_{\HH} + \e \bigr ]. 
	\end{align*}
	Iterating, we see that if for each $j = 1, \ldots, J(A)$, $\| \vec w(t_j) \|_{\HH} + \e < \delta_0$, then we have 
	\begin{align}\label{e325}
	\| w \|_{X(I_j)} \leq 10^J C_0^J [ \| \vec w(0) \|_{\HH} + \e \bigr ] \leq 10^J C_0^J \e < 10^J C_0^J \e_0.  
	\end{align}
	In particular, if we choose $\e_0 = \e_0(A)$ so small so that $10^J C_0^J \e_0 < \delta_0$, then the condition 
	$\| \vec w(t_j) \|_{\HH} + \e < \delta_0$ is always satisfied, and we conclude from \eqref{e325} that 
	\begin{align*}
	\| w \|_{X(I)} \leq C(A) \e,
	\end{align*}
	as desired.  
\end{proof}

\section{Concentration Compactness and Rigidity} \label{s:ccr}

In this section we outline the remainder of the proof of Theorem~\ref{t31}, equivalently Theorem~\ref{t:main}, which now proceeds via the Kenig-Merle concentration compactness/rigidity scheme~\cite{KM06, KM08}. The crucial ingredient in the concentration compactness step  are Bahouri-G\'erard-type profile decompositions~\cite{BG} and the rigidity scheme is based on exterior energy estimates developed by the first author along with Kenig and Schlag in~\cite{KLS}; note that the arguments in the latter reference were greatly influenced by the Duyckaerts, Kenig, Merle ideology from~\cite{DKM4, DKM5}. 


The implementation of this scheme in the setting of Adkins-Nappi wave maps was carried out in detail in~\cite{L15} and can be easily adapted to cover the more general situation we have now arrive at in the higher degree classes $n \ge 1$ with nontrivial $Q_n$.

\subsection{Concentration Compactness} 

As mentioned in Section 1, the proof of Theorem~\ref{t31} proceeds roughly as follows. We assume that Theorem~\ref{t31} fails and show that there must exist a so-called \emph{critical element}, which is a minimal solution to~\eqref{e31} satisfying~\eqref{eq:cond} that fails to scatter. The end--game here is the conclusion that the critical element possesses additional compactness properties that are then used to reach a contradiction in Section~\ref{s:r}. In particular one can prove the following.

\begin{prop}\label{p:cc}\cite[Corollary 4.5]{L15}  Suppose that Theorem~\ref{t31} fails. 
	Then there is  nonzero critical element $\vec u_{\infty}(t) \in \HH:= (\dot{H}^2 \times \dot{H}^1) \cap (\dot{H}^1 \times L^2) (\R^5)$ that is global in both time directions, solves~\eqref{e31}, and moreover 
	\EQ{
		\| u_{\infty} \|_{S( [0, + \infty))} = \| u_{\infty}\|_{S(( - \infty, 0])} = \infty,
	}
	and the set 
	\EQ{\label{K}
		\calK= \{  \vec u_{\infty}(t) \mid t \in \R \} \subset \HH
	}
	is pre-compact in $\HH$. 
\end{prop}

We omit the proof of Proposition~\ref{p:cc} and instead refer the reader to the detailed argument given in~\cite[Section 4]{L15} in the setting of degree $n=0$ Adkins--Nappi maps. The main technical tools are the small data theory established in Section 4 along with linear and nonlinear Bahouri--Ger\'ard-type profile decompositions adapted to setting of Adkins--Nappi maps.  The nonlinear profile decompositions are obtained from the linear profile decompositions by using Proposition \ref{p32} and Proposition \ref{p33}.  The only substantive difference between the situation here and~\cite{L15} is the presence of the potential $V$ on the left-hand-side of~\eqref{e31}. This gives rise to two different types of profiles at the level of the linear evolution (in \cite{L15} there are two different types of profiles only for the nonlinear profile decomposition; see ~\cite[Section 4.1.2]{L15}). A similar situation where multiple types of linear profiles arise has been treated in~\cite{LOS2} in the setting of non--radial waves with potential on hyperbolic space, and we refer the reader to the arguments in~\cite[Section 4.1]{LOS2} for how to include a potential $V$ into the linear profile decomposition. A straightforward implementation of arguments from~\cite{LOS2} (to deal with the potential) into the detailed scheme in~\cite{L15} covers the proof of Proposition~\ref{p:cc}.

%
%
%
%
%

\subsection{Rigidity}\label{s:r}

In this section we complete the proof of Theorem~\ref{t31}, equivalently
Theorem \ref{t:main}(b), by
showing that the critical element obtained  in Proposition~\ref{p:cc} cannot exist. 
In particular we prove the following rigidity result. 
\begin{prop}\label{rigid}Let $\vec u(t) \in \HH(\R^5)$ be a global solution to~\eqref{e31} and suppose that the trajectory 
	\EQ{
		K:=\{\vec u(t) \mid t \in \R\}
	}
	is pre-compact in $\HH(\R^5)$. Then $\vec u = (0, 0)$. 
\end{prop}

\begin{rem}
	In the original azimuth angle, $\vec \psi = \vec Q_n + r \vec u$, Proposition 
	\ref{rigid} states that the only solution $\vec \psi(t)$ to \eqref{eq:an} with precompact trajectory in 
	$\HH_n$ is the stationary Adkins--Nappi map $\vec Q_n$.  Since any stationary solution to \eqref{eq:an} has precompact trajectory in $\cl H_n$, we will in fact use, in an essential way, the uniqueness of $\vec Q_n$ as the \emph{only} stationary solution in the proof of Proposition \ref{rigid}. 
\end{rem}

A simple consequence of the precompactness of the trajectory is that the $\dot{H}^1 \times L^2$ norm of 
$\vec u(t)$ vanishes on any exterior cone $\{ r \ge R + \abs{t}\}$ as $\abs{t} \to \infty$ or $R \rar \infty$. 

\begin{cor}\label{vanish tail} Let $\vec u(t)$ and $K$ be as in Proposition~\ref{rigid}. Then we have
	\begin{align} 
	\begin{split}\label{compact ext en}
	\forall R > 0, \quad 
	\lim_{|t| \rar \infty} \| \vec u(t) \|_{\dot H^1 \times L^2( r \ge R+\abs{t})} &= 0, \\
	\lim_{R \rar \infty} \sup_{t \in \R} \| \vec u(t) \|_{\dot H^1 \times L^2(r \geq R + |t|)} &= 0
	\end{split}
	\end{align}
	where 
	\ant{
		\| \vec u(t) \|_{\dot H^1 \times L^2( r \ge R+\abs{t})}^2 := \int_{R + \abs{t}}^{\infty}( u_t^2(t, r) + u_r^2(t, r) )\, r^4 \, dr.
	}
\end{cor}

The proof will proceed in several steps following the rigidity argument given in~\cite{KLS}, and is similar to the one presented for the degree $0$ case in~\cite{L15}.
It draws 
from the  ``channels of energy" method pioneered in the seminal works \cite{DKM4, DKM5} 
on the $3d$ energy--critical wave equation. As in the degree $0$ case, a crucial ingredient in the proof is the
exterior energy estimate 
for free radial waves on $\R^{1+5}$ proved in~\cite{KLS}. 

\begin{prop}[Proposition 4.1, \cite{KLS}]\label{linear prop}
	Let $\Box w=0$ in $\R^{1+5}_{t,x}$ with radial data $(f,g)\in \dot H^1\times L^2(\R^5)$.
	Then with some absolute constant $c>0$ one has for every $a>0$
	\EQ{
		\label{R5ext}
		\max_{\pm}\;\limsup_{t\to\pm\I} \int_{r>a+|t|}^\I(w_t^2 + w_r^2)(t,r) r^4\, dr \ge c \| \pi_a^\perp (f,g)\|_{\dot H^1\times L^2(r>a)}^2
	}
	where $\pi_a=\Id-\pi_a^\perp$ is the orthogonal projection onto the plane $$P(a):=\{( c_1 r^{-3}, c_2 r^{-3})\:|\: c_1, c_2\in\R\}$$
	in the space $\dot H^1\times L^2(r>a)$.  The left-hand side of~\eqref{R5ext} vanishes for all data in this plane. 
\end{prop}
\begin{rem}One should note that the appearance of the projections $\pi^{\perp}_a$ on the 
	right-hand-side of~\eqref{R5ext} is due to the fact that $r^{-3}$ is the Newton potential in $\R^5$. 
	To be precise, consider initial data $(f, 0) \in \dot H^1 \times L^2 (r \ge R)$ which satisfies 
	$(f, 0) = (r^{-3}, 0)$ on $r \ge R>0$, with $f(r)$ vanishing on $r \le R/2$. 
	Then the corresponding free evolution $w(t, r)$ is given by $w(t, r) = r^{-3}$ on the 
	region $r \ge R+\abs{t}$ by finite speed of propagation. It is simple to see that the left-hand-side 
	of~\eqref{R5ext} vanishes for this solution as $t \to \pm \infty$. The other family of solutions projected away is generated by taking data
	$(0, g) = (0, r^{-3})$ on the exterior region $r\ge R>0$ which has solution $w(t, r) = tr^{-3}$ on 
	$r \ge R+ \abs{t}$. 
\end{rem}
\begin{rem}
	The orthogonal projections $\pi_a$, $\pi_a^{\perp}$ are given by 
	\ant{
		&\pi_a(f, 0) = a^3r^{-3} f(a), \quad 
		\pi_a(0, g) = a r^{-3} \int_a^{\infty} g( \rho)  \rho \, d \rho,\\
		&\pi_a^{\perp}(f, 0) = f(r) - a^3r^{-3} f(a), \quad \pi_a^{\perp}(0, g) = g(r) - a r^{-3} \int_a^{\infty} g( \rho)  \rho \, d \rho,
	}
	and thus we have 
	\ant{
		\|\pi_a(f,g)\|_{\dot{H}^1 \times L^2(r>a)}^2 &= 3a^3 f^2(a) + a\left(\int_a^{\infty} rg(r) \, dr\right)^2, \\
		\|\pi_a^{\perp}(f, g)\|_{\dot{H}^1 \times L^2(r>a)}^2&= \int^{\infty}_a f_r^2(r) \, r^4 \, dr - 3a^3f^2(a)\\
		& \quad + \int_a^{\infty} g^2(r) \, r^4 \, dr - a\left(\int_a^{\infty} rg(r) \, dr\right)^2.
	}
\end{rem}
\begin{rem}
	Through a delicate analysis Proposition \ref{linear prop} has been generalized to all odd dimensions in \cite{KLLS1}. The generalization was then used to establish stable 
	soliton resolution for all equivariant exterior wave maps in \cite{KLLS2} and later for all equivariant wave maps on a wormhole in \cite{CR16b}, \cite{CR16c}.  What makes the proof (and application) 
	of the generalization of Proposition \ref{linear prop} more difficult is that the subspace of 
	functions $P(a)$ projected away in the estimate grows with the dimension.
\end{rem}

The general idea is that the exterior energy decay~\eqref{compact ext en} can be combined with the exterior energy estimates for the underlying free equation in Proposition~\ref{linear prop}, to obtain precise {\em spacial} asymptotics for $u_0(r)= u(0, r)$  and $u_1(r) = u_t(0, r)$ as $r \to \infty$, namely, 
\EQ{\label{space asym}
	&r^3 u_0(r)  \to \ell_0 \mas r \to \infty,\\
	&r \int_r^{\infty} u_1( \rho) \rho \, d \rho  \to 0 \mas r \to \infty.
}
We then argue by contradiction to show that $\vec u(t) =(0, 0)$ is the only solution with both a 
pre-compact trajectory and initial data with the above asymptotics.  
We now begin the proof of Proposition~\ref{rigid}. 

\subsection{Step 1} In this first step, we will use Proposition~\ref{linear prop} along with Corollary~\ref{compact ext en} to derive 
the following uniform in time estimate on the projection $\pi^{\perp}_R \vec u(t)$.

\begin{lem} \label{project ineq}There exists $R_0>0$ so that for all $R>R_0$ and for all $t \in \R$ we have 
	\begin{align}
	\begin{split}\label{key ineq}
	\| \pi^{\perp}_R \vec u(t) \|_{\dot{H}^1 \times L^2(r \ge R)}
	\lesssim&\: R^{-4} \| \pi_R \vec u(t) \|_{\dot H^1 \times L^2(r \geq R)} \\ 
	&+R^{-5/2} \| \pi_R \vec u(t) \|^2_{\dot H^1 \times L^2(r \geq R)} \\
	&+R^{-1} \| \pi_R \vec u(t)\|_{\dot{H}^1 \times L^2(r \ge R)}^3
	\end{split}
	\end{align}
	where $P(R):= \{(c_1r^{-3}, c_2r^{-3}) \mid, c_1, c_2 \in \R\}$, $\pi_R$ denotes orthogonal projection onto $P(R)$ and $\pi_R^{\perp}$ denotes orthogonal projection onto the orthogonal complement of the plane $P(R)$ in the space $\dot{H}^1 \times L^2(r\geq R)(\R^5)$. We remark that the constant in~\eqref{key ineq} is uniform in $t \in \R$. 
\end{lem}

In order to prove Lemma~\ref{project ineq} we will first need a small data theory in the 
energy space $\dot H^1 \times L^2$
for a
modified Cauchy problem which is designed to capture the dynamics of our 
compact solution on the exterior cones $\CC_R:=\{(t, r) \mid r \ge R + \abs{t}\}$.  
The fact that we are only considering the evolution on the exterior cone $\CC_R$ allows us to 
truncate the initial data and the nonlinearity in a way that will render the initial value problem 
subcritical relative to the energy, while still preserving the flow on $\CC_R$ by finite speed of propagation. 

To accomplish this, we first fix a smooth function $\chi \in C^{\infty}([0, \infty))$ where $\chi(r) = 1$ 
for $r \ge 1$ and $\chi(r) = 0$ on $r \le 1/2$. 
We denote $\chi_R(r):= \chi(r/R)$ and for each $R>0$ we consider the modified Cauchy problem: 
\EQ{\label{h eq}
	&h_{tt} - h_{rr} - \frac{4}{r} h_r = Z_R(t, r, h),\\
	& Z_R(t, r, h) = -\chi_{R + |t|}(r) \bigl (V(r) h + Z(r,h)\bigr ),\\
	&\vec h(0) = (h_0, h_1) \in \dot H^1 \times L^2(\R^5),
}
where $V(r)$ and $Z(r,h)$ are defined as in \eqref{e32}--\eqref{e36}. 
The benefit of this modification is that forcing the nonlinearity to have support outside 
the ball of radius $R + |t|$ removes the super-critical nature of the problem which allows 
for a small data theory in $\dot{H}^1 \times L^2$ via Strichartz estimates and 
the usual contraction mapping based argument. In order to formulate the small data 
theory for~\eqref{h eq} we define the norm $X(I)$ where $0 \in I \subset \R$ is a time interval by 
\EQ{ \label{Z norm def}
	\|\vec h\|_{X(I)} = \|h\|_{L^{\frac{7}{3}}_t(I; L^{\frac{14}{3}}_x(\R^5))} + \|\vec h(t) \|_{L^{\infty}_t(I; \dot H^1 \times L^2)}
}

\begin{lem} \label{h small data}
	There exists a $\de_0 > 0$ small enough, so that for all  $R>0$ and all initial data $\vec h(0)=(h_0, h_1) \in \dot{H}^1 \times L^2 (\R^5)$ with 
	\ant{
		\| \vec h(0) \|_{ \dot H^1 \times L^2(\R^5)} < \de_0
	}
	there exists a unique global solution $\vec h(t) \in C(\R; \dot{H}^1 \times L^2$) to \eqref{h eq}. In addition $\vec h(t)$ satisfies 
	\EQ{\label{h small}
		\|\vec h\|_{X(\R)} \lesssim \|\vec h(0)\|_{\dot{H}^1 \times L^2} \lesssim  \de_0.
	}
	Moreover, if we denote the free evolution of the same data by $h_L(t):= S(t) \vec h(0)$, then we have 
	\begin{align}
	\begin{split}\label{h-hL}
	\sup_{t \in \R} \|  \vec h(t) - \vec h_L(t) \|_{ \dot{H}^1 \times L^2 } 
	\lesssim&\: R^{-4} \| \vec h(0) \|_{\dot H^1 \times L^2} 
	+ R^{-5/2} \| \vec h(0) \|^2_{\dot H^1 \times L^2} \\
	&+ R^{-1} \| \vec h(0) \|^3_{\dot H^1 \times L^2} + R^{-7/2} \| \vec h(0) \|^4_{\dot H^1 \times L^2} \\
	&+ R^{-4} \| \vec h(0) \|^5_{\dot H^1 \times L^2}.
	\end{split}
	\end{align}
\end{lem}

Along with Strichartz estimates and the Duhamel formula, the key ingredient in the proof of 
Lemma~\ref{h small data} is the Strauss Lemma for radial functions on $\R^5$: if $f \in \dot{H}^1(\R^5)$ 
then for each $r>0$ we have
\EQ{\label{Strauss}
	\abs{f (r)} \le Cr^{-3/2} \|f\|_{\dot{H}^1}.
}
\begin{proof}[Sketch of Proof of Lemma~\ref{h small data}] 
	The small data global well-posedness theory, including the estimate~\eqref{h small} 
	follows from the usual contraction and continuity arguments 
	based on the Strichartz estimates in Theorem~\ref{t:strich}(b). 
	The particular Strichartz estimate we use is  
	\EQ{
		\| v\|_{L^{\frac{7}{3}}_t L^{\frac{14}{3}}_x} + \|\vec v\|_{L^{\infty}_t(\dot{H}^1 \times L^2)}  
		\lesssim \| \vec v(0)\|_{\dot{H}^1 \times L^2} + \|F\|_{L^1_tL^2_x}
	}
	for solutions $\vec v $ to the inhomogenous $5d$ wave equation
	\begin{align*}
	&v_{tt} - \Delta v = F, \quad (t,x) \in R^{1+5}, \\
	&\vec v(0) = (v_0,v_1) \in \dot H^1 \times L^2.
	\end{align*}
	We omit the details and instead focus on the estimate~\eqref{h-hL} which has the same flavor.
	
	By the Strichartz estimate we have 
	\ant{
		\|\vec h(t)- \vec h_L(t)\|_{\dot{H}^1 \times L^2}  &\lesssim \| Z_R(\cdot, \cdot, h)\|_{L^1_t L^2_x}\\
		& \lesssim \| \chi_{R + |t|} V h \|_{L^1_t L^2_x} + \| \chi_{R + |t|} h^2 Z_2(r) \|_{L^1_t L^2_x} \\
		&\:\:+ \| \chi_{R + |t|} h^3 Z_3(r,rh) \|_{L^1_t L^2_x} + \| \chi_{R + |t|} h^4 Z_4(r,rh) \|_{L^1_t L^2_x} \\
		&\:\:+ \| \chi_{R + |t|} h^5 Z_5(rh) \|_{L^1_t L^2_x}.
	}
	We will now estimate each term in the right-hand-side above. By \eqref{pot bound}, H\"older's inequality and 
	simple calculations we have
	\begin{align*}
	\| \chi_{R+|t|} V h \|_{L^1_t L^2_x} &\lesssim 
	\| \chi_{R + |t|} r^{-6} h \|_{L^1_t L^2_x} \\
	&\lesssim \| \chi_{R + |t|} r^{-6} \|_{L^{7/4}_t L^{7/2}_x} \| h \|_{L^{7/3}_t L^{14/3}_x} \\ 
	&\lesssim \Bigl ( \int \Bigl ( \int_{\frac{1}{2}(R + |t|)}^{\infty}
	r^{-21} r^4 dr 
	\Bigr )^{1/2}
	\Bigr )^{4/7} \| h \|_{L^{7/3}_t L^{14/3}_x}  \\
	&\lesssim R^{-4} \| h \|_{X(\R)} \\ &\lesssim R^{-4} \| \vec h(0) \|_{\dot H^1 \times L^2}.
	\end{align*}
	
	We now estimate the nonlinear terms.  By \eqref{Z bounds} and H\"older's inequality we have 
	\begin{align*}
	\| \chi_{R + |t|} h^2 Z_2(r) \|_{L^1_t L^2_x} &\lesssim  
	\| \chi_{R + |t|} r^{-3} h^2 \|_{L^1_t L^2_x} \\
	&\lesssim \| \chi_{R + |t|} r^{-3} \|_{L^{7}_t L^{14}_x} \| h \|^2_{L^{7/3}_t L^{14/3}_x} \\
	&\lesssim R^{-5/2} \| h \|_{X(\R)}^2 \\
	&\lesssim R^{-5/2} \| \vec h(0) \|_{\dot H^1 \times L^2}^2. 
	\end{align*}
	By \eqref{Z bounds} and the Strauss estimate \eqref{Strauss} we have 
	\begin{align*}
	\| \chi_{R + |t|} h^3 Z_3(r,rh) \|_{L^1_t L^2_x} &\lesssim 
	\| \chi_{R + |t|} h^3 \|_{L^1_t L^2_x} \\
	&\lesssim \Bigl ( \sup_{t,r} |\chi_{R + |t|} h| \Bigr )^{2/3} \| h \|^{7/3}_{L^{7/3}_t L^{14/3}_x} \\
	&\lesssim R^{-1} \| \vec h \|_{L^\infty_t(\dot H^1 \times L^2)}^{2/3} \| h \|^{7/3}_{L^{7/3}_t L^{14/3}_x} \\
	&\lesssim R^{-1} \| h \|_{X(\R)}^3 \\
	&\lesssim R^{-1} \| \vec h(0) \|_{\dot H^1 \times L^2}^3. 
	\end{align*}
	In a similar fashion, we estimate the quartic and quintic terms: 
	\begin{align*}
	\| \chi_{R + |t|} h^4 Z_4(r,rh) \|_{L^1_t L^2_x} &\lesssim 
	\| \chi_{R + |t|} r^{-1} h^4 \|_{L^1_t L^2_x} \\
	&\lesssim R^{-1} \Bigl ( \sup_{t,r} |\chi_{R + |t|} h| \Bigr )^{5/3} \| h \|^{7/3}_{L^{7/3}_t L^{14/3}_x} \\
	&\lesssim R^{-7/2} \| \vec h \|_{L^\infty_t(\dot H^1 \times L^2)}^{5/3} \| h \|^{7/3}_{L^{7/3}_t L^{14/3}_x} \\
	&\lesssim R^{-7/2} \| h \|_{X(\R)}^4 \\
	&\lesssim R^{-7/2} \| \vec h(0) \|_{\dot H^1 \times L^2}^4, 
	\end{align*}
	and 
	\begin{align*}
	\| \chi_{R + |t|} h^5 Z_5(rh) \|_{L^1_t L^2_x} &\lesssim 
	\| \chi_{R + |t|} h^5 \|_{L^1_t L^2_x} \\
	&\lesssim \Bigl ( \sup_{t,r} |\chi_{R + |t|} h| \Bigr )^{8/3} \| h \|^{7/3}_{L^{7/3}_t L^{14/3}_x} \\
	&\lesssim R^{-4} \| \vec h \|_{L^\infty_t(\dot H^1 \times L^2)}^{8/3} \| h \|^{7/3}_{L^{7/3}_t L^{14/3}_x} \\
	&\lesssim R^{-4} \| h \|_{X(\R)}^5 \\
	&\lesssim R^{-4} \| \vec h(0) \|_{\dot H^1 \times L^2}^5. 
	\end{align*}
	This completes the proof of \eqref{h-hL}.
\end{proof}

\begin{rem}\label{u=uR}
	We remark that for every $t \in \R$ the nonlinearity $Z_R$ in~\eqref{h eq} satisfies
	\ant{
		Z_R(t,r, u) = - V(r) u - Z(r,u), \quad \forall r \ge R + \abs{t}.
	}
	Thus, by finite speed of propagation we can conclude that solutions to~\eqref{h eq} and~\eqref{e31} 
	are equal on the exterior cone $\CC_R:=\{(t, r) \mid r \ge R + \abs{t}\}$.
\end{rem}
We now prove Lemma~\ref{project ineq}. 
\begin{proof}[Proof of Lemma~\ref{project ineq}] We will prove the Lemma first for time $t=0$. The proof for all times $t \in \R$ with $R>R_0$ independent of $t$ will follow from the pre-compactness of the trajectory, 
	in particular Corollary \ref{vanish tail}. We begin by defining truncated initial data, $\vec u_R(0) = (u_{0, R}, u_{1, R})$ by
	\EQ{
		&u_{0, R}(r):= \begin{cases} u_0(r) \mfor r \ge R\\ u_0(R) \mfor 0 \le r \le R \end{cases},\\
		&u_{1, R}(r):= \begin{cases} u_1(r) \mfor r \ge R\\ 0 \mfor 0 \le r \le R \end{cases}.
	}
	This truncated initial data satisfies  
	\EQ{
		\| \vec u_R(0)\|_{\dot{H}^1 \times L^2} = \|\vec u(0) \|_{\dot{H}^1 \times L^2(r \ge R)}
	}
	and so we can choose $R_0>0$ large enough so that for all $R \ge R_0$ we have 
	\EQ{
		\| \vec u_R(0)\|_{\dot{H}^1 \times L^2} \le \de \le \min(\de_0, 1)
	}
	where $\de_0$ is chosen as in Lemma~\ref{h small data}. By Lemma~\ref{h small data} there exists a unique solution $\vec u_R(t)$ to~\eqref{h eq} with initial data $\vec u_R(0)$ which satisfies~\eqref{h small} and \eqref{h-hL}. We note that by finite speed of propagation we have 
	\EQ{\label{u=uR eq}
		\vec u_R(t,r) = \vec u(t,r), \quad \forall (t, r) \in \CC_R.
	}
	Let $\vec u_{R, L}(t):= S(t) \vec u_{R}(0)$ be the free evolution with initial data $\vec u_R(0)$. By the triangle inequality and~\eqref{u=uR eq}, we have 
	\ant{
		\| \vec u(t) \|_{\dot{H}^1\times L^2(r \ge R+\abs{t})} &= \| \vec u_R(t) \|_{\dot{H}^1\times L^2(r \ge R+\abs{t})} \\
		&\ge \|u_{R, L}(t) \|_{\dot{H}^1\times L^2(r \ge R+\abs{t})} \\&\:\:- \|\vec u_R(t)- u_{R, L}(t) \|_{\dot{H}^1\times L^2(r \ge R+\abs{t})} 
	}
	Applying~\eqref{h-hL} to $\vec u_{R}(t)$ and taking $R>R_0$ large enough,  we can conclude 
	\ant{
		\|\vec u_R(t)- u_{R, L}(t) \|_{\dot{H}^1\times L^2(r \ge R+\abs{t})}  &\le  \|\vec u_R(t)- u_{R, L}(t) \|_{\dot{H}^1\times L^2} \\
		&\lesssim R^{-4} \| \vec u_{R}(0)\|_{\dot{H}^1 \times L^2} +R^{-5/2} \| \vec u_{R}(0)\|_{\dot{H}^1 \times L^2}^2 \\
		&\:\:+
		R^{-1} \| \vec u_{R}(0)\|_{\dot{H}^1 \times L^2}^3 +R^{-7/2} \| \vec u_{R}(0)\|_{\dot{H}^1 \times L^2}^4\\
		&\:\: + R^{-4} \| \vec u_R(0) \|_{\dot H^1 \times L^2}^5\\
		&\lesssim R^{-4} \| \vec u(0)\|_{\dot{H}^1 \times L^2(r \geq R)} \\ &\:\:+R^{-5/2} \| \vec u(0)\|_{\dot{H}^1 \times L^2(r \geq R)}^2 \\
		&\:\:+
		R^{-1} \| \vec u(0)\|_{\dot{H}^1 \times L^2(r \geq R)}^3.
	}
	Combining the previous two inequalities gives
	\begin{align*}
	\| \vec u(t) \|_{\dot{H}^1\times L^2(r \ge R+\abs{t})} &\ge  \|u_{R, L}(t) \|_{\dot{H}^1\times L^2(r \ge R+\abs{t})} - C_0R^{-4}\| \vec u(0)\|_{\dot{H}^1 \times L^2(r \ge R)} \\
	&\:\:- C_0R^{-5/2}\| \vec u(0)\|^2_{\dot{H}^1 \times L^2(r \ge R)} 
	- C_0R^{-1}\| \vec u(0)\|^3_{\dot{H}^1 \times L^2(r \ge R)}.
	\end{align*}
	Letting $t \to \pm \infty$ -- the choice determined by Proposition~\ref{linear prop} -- we can use Proposition~\ref{linear prop} to give a lower bound for the right-hand side and use Corollary~\ref{vanish tail} to see that the left-hand-side above goes to zero as $\abs{t} \to \infty$ and deduce the estimate
	\begin{align*}
	\| \pi^{\perp}_R  \vec u_{R}(0) \|_{ \dot{H}^1 \times L^2(r \ge R)} &\lesssim R^{-4}\| \vec u(0)\|_{\dot{H}^1 \times L^2(r \ge R)}
	+ R^{-5/2}\| \vec u(0)\|_{\dot{H}^1 \times L^2(r \ge R)}^2 \\
	&\:\:+ R^{-1}\| \vec u(0)\|_{\dot{H}^1 \times L^2(r \ge R)}^3.
	\end{align*}
	Since $\vec u_R(0)= \vec u(0)$ on $\{r \ge R\}$, the above implies that  
	\begin{align*}
	\| \pi^{\perp}_R  \vec u(0) \|_{ \dot{H}^1 \times L^2(r \ge R)} &\lesssim R^{-4}\| \vec u(0)\|_{\dot{H}^1 \times L^2(r \ge R)}
	+ R^{-5/2}\| \vec u(0)\|_{\dot{H}^1 \times L^2(r \ge R)}^2 \\
	&\:\:+ R^{-1}\| \vec u(0)\|_{\dot{H}^1 \times L^2(r \ge R)}^3.
	\end{align*}
	Finally, we can use the orthogonality of $\pi_R$
	\begin{align*}
	\| \vec u(0) \|_{\dot H^1 \times L^2(r \geq R)}^2 
	=  \| \pi_R \vec u(0) \|_{\dot H^1 \times L^2(r \geq R)}^2 + 
	+  \| \pi_R^{\perp} \vec u(0) \|_{\dot H^1 \times L^2(r \geq R)}^2 
	\end{align*}
	to expand the right-hand side of the previous estimate and choose $R_0$ large enough so that we can absorb the $\pi_R^{\perp}$ terms on the right-hand-side into the left-hand-side to deduce that 
	\begin{align*}
	\begin{split}
	\| \pi^{\perp}_R \vec u(0) \|_{\dot{H}^1 \times L^2(r \ge R)}
	\lesssim&\: R^{-4} \| \pi_R \vec u(0) \|_{\dot H^1 \times L^2(r \geq R)} \\ 
	&+R^{-5/2} \| \pi_R \vec u(0) \|^2_{\dot H^1 \times L^2(r \geq R)} \\
	&+R^{-1} \| \pi_R \vec u(0)\|_{\dot{H}^1 \times L^2(r \ge R)}^3.
	\end{split}
	\end{align*}
	This proves the lemma for $t=0$. To see that the inequality holds for all $t \in \R$ we note that by the pre-compactness of $K$ we can choose $R_0 = R_0( \de_0)$ so that  for all $R \ge R_0$ we have 
	\ant{
		\| \vec u(t) \|_{ \dot{H}^1 \times L^2( r \ge R)} \le \min(\de_0, 1)
	} 
	uniformly in $t \in \R$. Now we just repeat the entire argument above with truncated initial data at time $t=t_0$ and $R \ge R_0$ given by 
	\ant{
		&u_{0, R, t_0}(r):= \begin{cases} u(t_0, r) \mfor r \ge R\\ u_0(t_0, R) \mfor 0 \le r \le R \end{cases},\\
		&u_{1, R, t_0}(r):= \begin{cases} u_t(t_0, r) \mfor r \ge R\\ 0 \mfor 0 \le r \le R \end{cases}.
	}
	This finishes the proof. 
\end{proof}

\subsection{Step 2}
Next, we will use the estimates in Lemma~\ref{project ineq} to prove the spatial asymptotics of $\vec u(0, r)$ as $r \to \infty$ described in~\eqref{space asym}. The result of this step is the following.

\begin{lem}\label{spacial decay}
	Let $\vec u(t)$ be as in Proposition~\ref{rigid}. Then there exists $\ell_0 \in \R$ so that 
	\begin{align}
	r^3 u_0(r) \to \ell_0 \mas r \to \infty, \label{u0 decay}\\
	r \int_r^\I u_1( \rho) \rho \, d \rho \to 0 \mas r \to \infty. \label{u1 decay}
	\end{align}
	Moreover the above convergence occurs at the rates 
	\begin{align}
	\label{u0 rate}\abs{r^3 u_0(r) - \ell_0}  = O(r^{-4}) \mas r \to \infty, \\
	\label{u1 rate}\abs{r \int_r^\I u_1( \rho) \rho \, d \rho } = O(r^{-2}) \mas r \to \infty.
	\end{align}
\end{lem}
For the proof of Lemma~\ref{spacial decay}, we first define the following 
quantities: 
\EQ{\label{v def}
	&v_0(t, r):= r^3 u(t, r), \\
	&v_1(t, r):= r \int_r^\I u_t( t, \rho) \rho \, d \rho.
}
To simplify notation, we will often write $v_0(r) :=v_0(0, r)$ and $v_1(r):=v_1(0, r)$. With these definitions we have 
\EQ{\label{v u project} 
	&\| \pi_R \vec u(t)\|^2_{\dot{H}^1 \times L^2(r \ge R)} = 3R^{-3} v_0^2(t, R) + R^{-1}v_1^2(t, R),\\
	&\| \pi_R^{\perp} \vec u(t)\|^2_{\dot{H}^1 \times L^2(r \ge R)} = \int_R^\I \left( \frac{1}{r} \p_r v_0(t, r) \right)^2 \, dr + \int_R^\I \left( \p_r v_1(t ,r) \right)^2 dr.
}
Moreover, the conclusions of Lemma \ref{spacial decay} we wish to prove can be rewritten as 
\begin{align*}
|v_0(t,r) - \ell_0| &= O(r^{-4}) \mas r \rar \infty, \\
|v_1(t,r)| &= O(r^{-2}) \mas r \rar \infty, 
\end{align*}
where the $O(\cdot)$ terms are uniform in $t \in \R$. 

First, we rewrite the conclusions of Lemma~\ref{space asym} in terms of $(v_0, v_1)$. 
\begin{lem}\label{v project}
	Let $(v_0, v_1)$ be defined as in~\eqref{v def}. Then there exists $R_0>0$ such that for all $R \ge R_0$ we have 
	\ant{
		\int_R^\I & \left( \frac{1}{r} \p_r v_0(t, r) \right)^2 + \left( \p_r v_1(t ,r) \right)^2 \, dr \\ &\lesssim 
		R^{-8} \left(3R^{-3} v_0^2(t, R) + R^{-1}v_1^2(t, R)\right) \\
		&\:\:+ R^{-5} \left(3R^{-3} v_0^2(t, R) + R^{-1}v_1^2(t, R)\right)^2 \\
		&\:\:+ R^{-2} \left(3R^{-3} v_0^2(t, R) + R^{-1}v_1^2(t, R)\right)^3 \\
		&\lesssim R^{-11} |v_0(t,R)|^2 + 
		R^{-11} |v_0(t,R)|^4 + R^{-11} |v_0(t,R)|^6 \\
		&\:\:+ R^{-9} |v_1(t,R)|^2 + 
		R^{-7} |v_1(t,R)|^4 + R^{-5} |v_1(t,R)|^6 
	}
	with a constant that is uniform in $t \in \R$. 
\end{lem}

We will use Lemma~\ref{v project} to prove difference estimates. Let $\de_1>0$ be small, to be determined below, so that $\de_1 \le \de_0^2$ where $\de_0$ is as in the small data theory in Lemma~\ref{h small data}. Let $R_1 = R_1(\de _1)$ be large enough so that for all $R \ge R_1$ we have 
\EQ{\label{R1 de1}
	&\| \vec u(t) \|_{\dot{H}^1 \times L^2(r \ge R)}^2 \le  \de_1 \le \de_0^2 \quad \forall t,\\
	&R_1^{-3} \le  \de_1.
}
Such an $R_1$ exists by Corollary \ref{vanish tail}. Note that by \eqref{v u project} and \eqref{R1 de1} we have for all $R \geq R_1$
\begin{align}
\begin{split}\label{v grow1}
|v_0(t,R)| &\lesssim R^{3/2} \de_1, \\
|v_1(t,R)| &\lesssim R^{1/2} \de_1,
\end{split}	
\end{align}
uniformly in $t \in \R$. 

\begin{cor}\label{diff esti}
	Let $R_1$ be as above. Then for all $R_1 \le r \le r' \le 2r$ and for all $t \in \R$  we have 
	\EQ{\label{v0 diff1}
		\abs{ v_0(t, r) - v_0(t, r') } &\lesssim r^{-4} |v_0(t,r)| + 
		r^{-4} |v_0(t,r)|^2 + r^{-4} |v_0(t,r)|^3 \\
		&\:\:+ r^{-3} |v_1(t,r)| + 
		r^{-2} |v_1(t,r)|^2 + r^{-1} |v_1(t,r)|^3 
	} 
	and
	\EQ{\label{v1 diff1}
		\abs{ v_1(t, r) - v_1(t, r') } &\lesssim r^{-5} |v_0(t,r)| + 
		r^{-5} |v_0(t,r)|^2 + r^{-5} |v_0(t,r)|^3 \\
		&\:\:+ r^{-4} |v_1(t,r)| + 
		r^{-3} |v_1(t,r)|^2 + r^{-2} |v_1(t,r)|^3 
	}
	with the above estimates holding uniformly in $t \in \R$. 
\end{cor}
From Corollary~\ref{diff esti}, the definitions of $\de_1$ and $R_1 = R_1( \de_1)$ and \eqref{v grow1} we have the following simple consequence.

\begin{cor} \label{diff esti 2}
	Let $R_1, \de_1$ be as in~\eqref{R1 de1}. Then for all $r, r'$ with $R_1 \le r \le r' \le 2r$ and for all $t \in \R$ we have 
	\EQ{\label{v0 diff2}
		\abs{ v_0(t, r) - v_0(t, r') } \lesssim r^{-1} \de_1 \abs{ v_0(t, r)}+ \de_1 \abs{ v_1(t, r)}
	} 
	and
	\EQ{\label{v1 diff2}
		\abs{ v_1(t, r) - v_1(t, r') } \lesssim r^{-2} \de_1 \abs{ v_0(t, r)} + r^{-1} \de_1 \abs{ v_1(t, r)}
	} 
	with the above estimates holding uniformly in $t \in \R$. 
\end{cor}

\begin{proof}[Proof of Corollary~\ref{diff esti}]
	This will follow from Lemma~\ref{v project} and the fundamental theorem of 
	calculus. For $r \ge R_1$ and $r \le r' \le 2r$ we see that 
	\ant{
		\abs{ v_0(t, r) - v_0(t, r') } &\le \left( \int_r^{r'} \abs{\p_r v_0(t,  \rho) } \, d \rho\right)\\
		& \le  \left( \int_r^{r'} \abs{\frac{1}{\rho}\p_r v_0(t,  \rho) }^2 \, d \rho\right)^{\frac{1}{2}} \left( \int_r^{r'} \rho^2 \, d \rho\right)^{\frac{1}{2}}\\
		&\lesssim r^{\frac{3}{2}} \Bigl [ r^{-\frac{11}{2}} |v_0(t,r)| + 
		r^{-\frac{11}{2}} |v_0(t,r)|^2 + r^{-\frac{11}{2}} |v_0(t,r)|^3 \\
		&\:\:\:\:+ r^{-\frac{9}{2}} |v_1(t,r)| + 
		r^{-\frac{7}{2}} |v_1(t,r)|^2 + r^{-\frac{5}{2}} |v_1(t,r)|^3 \Bigr ] \\
		&\lesssim r^{-4} |v_0(t,r)| + 
		r^{-4} |v_0(t,r)|^2 + r^{-4} |v_0(t,r)|^3 \\
		&\:\:+ r^{-3} |v_1(t,r)| + 
		r^{-2} |v_1(t,r)|^2 + r^{-1} |v_1(t,r)|^3 
	}
	where in the second to last inequality above we used Lemma~\ref{v project}. Similarly, 
	\ant{
		\abs{ v_1(t, r) - v_1(t, r') } &\le \left( \int_r^{r'} \abs{\p_r v_1(t,  \rho) } \, d \rho\right)\\
		& \le  \left( \int_r^{r'} \abs{\p_r v_1(t,  \rho) }^2 \, d \rho\right)^{\frac{1}{2}} \left( \int_r^{r'}  \, d \rho\right)^{\frac{1}{2}}\\
		&\lesssim r^{\frac{1}{2}}  \Bigl [ r^{-\frac{11}{2}} |v_0(t,r)|^1 + 
		r^{-\frac{11}{2}} |v_0(t,r)|^2 + r^{-\frac{11}{2}} |v_0(t,r)|^3 \\
		&\:\:\:\:+ r^{-\frac{9}{2}} |v_1(t,r)| + 
		r^{-\frac{7}{2}} |v_1(t,r)|^2 + r^{-\frac{5}{2}} |v_1(t,r)|^3 \Bigr ] \\
		&\lesssim r^{-5} |v_0(t,r)| + 
		r^{-5} |v_0(t,r)|^2 + r^{-5} |v_0(t,r)|^3 \\
		&\:\:+ r^{-4} |v_1(t,r)| + 
		r^{-3} |v_1(t,r)|^2 + r^{-2} |v_1(t,r)|^3 
	}
	as desired.
\end{proof}
Using these difference estimates we will prove bounds on the growth rates for $v_0(t, r)$ and $v_1(t, r)$ as $r \to \infty$ which are improvements of 
\eqref{v grow1}. 
\begin{claim}
	Let $v_0(t, r)$ and $v_1(t, r)$ be as in~\eqref{v def}. Then 
	\begin{align}
	\label{v0 grow} &\abs{v_0(t, r)} \lesssim r^{\frac{1}{6}}, \\
	\label{v1 grow} &\abs{v_1(t, r)} \lesssim r^{\frac{1}{6}},
	\end{align}
	where again the constants above are uniform in $t \in \R$. 
\end{claim}
\begin{proof}
	We begin by noting that it suffices to prove the claim for $t=0$ since the argument relies solely on estimates in this section that hold uniformly in $t \in\R$. 
	
	Fix $r_0 \ge R_1$ and observe that by setting $r= 2^nr_0$, $r'=2^{n+1}r_0$ in the difference estimates~\eqref{v0 diff2},~\eqref{v1 diff2} we have for each $n \in \N$, 
	\begin{align}
	&\label{v0 tri}\abs{v_0(2^{n+1} r_0)} \le (1+ C_1(2^n r_0)^{-1} \de_1) \abs{v_0(2^nr_0)} + C_1 \de_1 \abs{v_1(2^nr_0)},\\
	&\label{v1 tri}\abs{v_1(2^{n+1} r_0)} \le (1+ C_1(2^n r_0)^{-1} \de_1) \abs{v_1(2^nr_0)} + C_1 \de_1(2^n r_0)^{-2} \abs{v_0(2^nr_0)}.
	\end{align}
	We introduce the notation 
	\begin{align}
	a_n:= \abs{v_0(2^nr_0)},\\
	b_n:=\abs{v_1(2^nr_0)}.
	\end{align}
	Adding~\eqref{v0 tri} with \eqref{v1 tri} yields
	\ant{
		a_{n+1} + b_{n+1} &\le (1+ C_1\de_1((2^n r_0)^{-1} +(2^n r_0)^{-2}))a_n + (1+C_1 \de_1(1+(2^n r_0)^{-1}))b_n\\
		&\le (1+ 2C_1 \de_1)(a_n + b_n).
	}
	Arguing inductively we see that for each $n$, 
	\ant{
		(a_n+b_n) \le (1+2C_1 \de_1)^{n}(a_0+b_0).
	}
	We now choose $\de_1$ small enough so that $(1+2C_1 \de_1) \le 2^{\frac{1}{6}}$. This allows us to conclude that 
	\EQ{\label{an bn}
		a_n \le C(2^n r_0)^{\frac{1}{6}},\\
		b_n \le C(2^n r_0)^{\frac{1}{6}},
	}
	where the constant $C=C(r_0)$ which is irrelevant since $r_0$ is fixed. Note that~\eqref{an bn} proves~\eqref{v0 grow} and~\eqref{v1 grow} for $r=2^nr_0$. The general estimates ~\eqref{v0 grow} and~\eqref{v1 grow} follow easily from ~\eqref{an bn} and the difference estimates~\eqref{v0 diff1},~\eqref{v1 diff1}. 
\end{proof}

With the improved growth rates of $(v_0, v_1)$ proved in the previous claim and our difference estimates \eqref{v0 diff1}, \eqref{v1 diff1}, we can now extract spatial limits. We begin with $v_1(t, r)$ as we will need to first show that this tends to $0$ in  order to get the correct rate for $v_0(t, r)$.

\begin{claim} \label{ell1 claim}For each $t \in \R$ there exists $\ell_1(t) \in \R$ so that 
	\EQ{
		\abs{ v_1(t, r) -  \ell_1(t)}  = O(r^{-2}) \mas r \to \infty
	}
	where the $O( \cdot)$ above is uniform in $t \in \R$. 
\end{claim}
\begin{proof}
	As usual, we only need to give the proof for $t=0$. Let $r_0 \ge R_1$ where $R_1 $ is as in~\eqref{R1 de1}. Plugging~\eqref{v0 grow},~\eqref{v1 grow} into the difference estimates~\eqref{v1 diff1} gives 
	\begin{align}
	\begin{split}\label{v1 conv}
	\abs{ v_1(2^{n+1}r_0) - v_1(2^n r_0) } &\lesssim 
	(2^n r_0)^{-5} (2^n r_0)^{1/6} + 
	(2^n r_0)^{-5} (2^n r_0)^{1/3} \\&\:\:+ (2^n r_0)^{-5} (2^n r_0)^{1/2} \\
	&\:\:+ (2^n r_0)^{-4} (2^n r_0)^{1/6} + 
	(2^n r_0)^{-3} (2^n r_0)^{1/3} \\&\:\:+ (2^n r_0)^{-2} (2^n r_0)^{1/2} \\
	&\lesssim (2^n r_0)^{-3/2}.
	\end{split}
	\end{align}
	Therefore we have 
	\ant{
		\sum_n \abs{ v_1(2^{n+1}r_0) - v_1(2^n r_0) } < \infty,
	}
	which implies there exists $\ell_1 \in \R$ so that 
	\ant{
		\lim_{n \to \infty} v_1(2^nr_0) = \ell_1.
	}  
	Moreover, we have the bound
	\begin{align*}
	|v_1(2^n r_0) - v_1(r_0)| &\leq  \sum_n \abs{ v_1(2^{n+1}r_0) - v_1(2^n r_0) }\\ &\lesssim \sum_n (2^n r_0)^{-3/2} \\
	&\leq C_0(u,r_0).	
	\end{align*}
	In particular, we have that the sequence $(v_0(2^n r_0))_n$ is bounded by a constant depending only on $u$ and $r_0$.  Fix $r \geq r_0$ with 
	$2^n r_0 \leq r < 2^{n+1}r_0$. The boundedness of the sequence $(v_1(2^n r_0))$ and the difference estimate \eqref{v1 diff1} imply the improvement of \eqref{v1 conv}
	\begin{align*}
	|v_1(2^{m+1}r_0) - v_1(2^m r_0)| \leq C_1(u,r_0) (2^m r_0)^{-2}, \quad \forall m,
	\end{align*}
	as well as 
	\begin{align*}
	|v_1(r) - v_1(2^n r_0)| \leq C_2(u,r_0) (2^n r_0)^{-2}.
	\end{align*}
	Combining these inequalities, we conclude
	\begin{align*}
	|\ell_1 - v_1(r)| &\leq \sum_{m = n} |v_1(2^{m+1} r_0) - v_1(2^m r_0)| + 
	|v_1(2^n r_0) - v_1(r)| \\
	&\leq C_3(u,r_0) (2^n r_0)^{-2} \\
	&\leq C_3(u,r_0) r^{-2}
	\end{align*}
	which proves convergence for all $r$ along with the rate of convergence. 
\end{proof}

Next, we prove that the limit $\ell_1(t)$ is independent of $t$. 
\begin{claim}\label{ell1 constant} The function $\ell_1(t)$ in Claim~\ref{ell1 claim} is independent of $t$, i.e. $\ell_1(t) =  \ell_1$ for all $t \in \R$.
	
\end{claim}

\begin{proof} Recall that, by definition 
	\ant{
		v_1(t, r) = r\int_r^{\I} u_t(t, \rho) \, \rho \, d\rho.
	}
	By Claim~\ref{ell1 claim}, we then have 
	\ant{
		\ell_1(t) =  r\int_r^{\I} u_t(t, \rho) \, \rho \, d\rho + O(r^{-2}) \mas r \to \infty
	}
	Now, let $t_1, t_2 \in \R$ with  $t_1 \neq t_2$. Then 
	\ant{
		\ell_1(t_2) - \ell_1(t_1) 
		&= \frac{1}{R} \int_R^{2R}\left(s \int_s^{\infty} (u_t(t_2, r) - u_t(t_1, r)) r \, dr \right) \, ds + O(R^{-2})\\
		&=\frac{1}{R} \int_R^{2R}\left(s \int_s^{\infty}  \int_{t_1}^{t_2}u_{tt}(t, r) \, dt\,  r \, dr \right) \, ds + O(R^{-2}).
	}
	Using the fact that $\vec u(t)$ is a solution to \eqref{e31}
	we can rewrite the above integral as
	\ant{
		&=  \int_{t_1}^{t_2}\frac{1}{R} \int_R^{2R}\left(s \int_s^{\infty}  (ru_{rr}(t, r) + 4u_r(t, r) ) \, dr \right) \, ds\, dt + \\
		&\, -  \int_{t_1}^{t_2} \frac{1}{R} \int_R^{2R}\left(s \int_s^{\infty}V(r)u(t,r) + Z(r, u(t,r))   \, dr \right) \,ds\, dt  + O(R^{-2})\\
		& = A + B +  O(R^{-2}).
	}
	To estimate $A$ we integrate by parts twice: 
	\EQ{ \label{I}
		A&= \int_{t_1}^{t_2}\frac{1}{R} \int_R^{2R}\left(s \int_s^{\infty}  \frac{1}{r^3} \p_r(r^4 u_r(t, r)) \, dr \right) \, ds\, dt\\
		& = 3\int_{t_1}^{t_2}\frac{1}{R} \int_R^{2R}\left(s \int_s^{\infty}  u_r(t, r) \, dr \right) \, ds\, dt -\int_{t_1}^{t_2}\frac{1}{R} \int_R^{2R}s^2  u_r(t, s) \, ds\, dt \\
		& = -3\int_{t_1}^{t_2}\frac{1}{R} \int_R^{2R}r \, u(t, r) \, dr\, dt -\int_{t_1}^{t_2}\frac{1}{R} \int_R^{2R}r^2\, u_r(t, r) \, dr\, dt \\
		&= -\int_{t_1}^{t_2}\frac{1}{R} \int_R^{2R}r \, u(t, r) \, dr\, dt + \int_{t_1}^{t_2} (Ru(t,R)-2Ru(t, 2R))\, dt.
	}
	By the definition of $v_0(t, r)$ and ~\eqref{v0 grow} we have 
	\EQ{\label{u control}
		r^3\abs{u(t, r)} = \abs{v_0(t, r))} \lesssim r^{\frac{1}{6}}
	}
	uniformly in $t \in \R$, and hence $\abs{u(t, r)} \lesssim r^{-\frac{17}{6}}$ uniformly in $t \in\R$. Plugging this into \eqref{I} we obtain
	\ant{
		A = \abs{t_2-t_1}O(R^{-\frac{11}{6}})
	}
	To estimate $B$ we use~\eqref{u control}, \eqref{pot bound} and \eqref{Z bounds} to see that for $r>R_1$ we have 
	\begin{align*}
	\abs{V(r) u(t,r) + Z(r, u(t, r))} &\lesssim r^{-53/6} + r^{-26/3} + r^{-17/2}\\ &\:\:+ r^{-37/3} + r^{-85/6} \\
	&\lesssim r^{-8}.
	\end{align*} 
	Therefore, 
	\ant{
		B \lesssim \int_{t_1}^{t_2} \frac{1}{R} \int_R^{2R}\left(s \int_s^{\infty}r^{-8}   \, dr \right) \,ds\, dt = \abs{t_2- t_1} O(r^{-6}).
	}
	Combining the estimates for $A$ and $B$ we obtain
	\ant{
		\abs{\ell_1(t_2)- \ell_1(t_1)}  = O(R^{-\frac{11}{6}})|t_1 - t_2| + 
		O(R^{-2}) \mas R \to \infty
	}
	which implies that $\ell_1(t_1) = \ell_1(t_2)$, i.e. $\ell_1(t)$ is independent of time. 
\end{proof}
We will now show that $\ell_1 = 0$. 
\begin{claim}\label{ell1=0}  $\ell_1 =0$. 
\end{claim} 
\begin{proof} 
	By Claim \ref{ell1 claim} we know that for all $R \ge R_1$  and for all $t \in \R$ we have 
	\ant{
		R\int_R^{\infty}u_t(t, r) \, r \, dr = \ell_1 + O(R^{-2}),
	}
	where $O(\cdot)$ is uniform in~$t$.
	Integrating from $t=0$ to $t=T$ and dividing by $T$ gives
	\begin{align*}
	\ell_1 &= \frac{R}{T} \int_R^\infty \int_0^T u_t(t,r) dt r dr + O(R^{-2}) \\
	&= \frac{R}{T} \int_R^\infty [u(T,r) - u(0,r)] r dr + O(R^{-2}).
	\end{align*}
	By \eqref{u control} we have that 
	\begin{align*}
	\frac{R}{T} \left | \int_R^{\infty}(u(T, r) - u(0,r) )\, r \, dr \right | 
	& \lesssim \frac{R}{T} \int_R^{\infty} r^{-\frac{11}{6}}\, dr \\
	&\lesssim \frac{R^{\frac{1}{6}}}{T}.
	\end{align*}
	Therefore  we have for all $R \geq R_1$ and $T > 0$
	\begin{align*}
	|\ell_1| \lesssim \frac{R^{1/6}}{T} + R^{-2}.
	\end{align*}
	By choosing $T = R$ and letting $R \rar \infty$, we conclude $\ell_1 = 0$ as desired.   
\end{proof}

We can finish the proof of Lemma~\ref{spacial decay}. 
\begin{proof}[Proof of Lemma~\ref{spacial decay}]
	We note that combining the results of Claims~\ref{ell1 claim},~\ref{ell1 constant} and \ref{ell1=0}, we have established~\eqref{u1 decay} and~\eqref{u1 rate}, namely that 
	\EQ{\label{v1 to 0}
		\abs{v_1(r)} = O(r^{-2}) \mas r \to 0.
	}
	It therefore remains to show that there exists $\ell_0 \in \R$ so that 
	\EQ{
		\abs{v_0(r) - \ell_0} = O(r^{-4}) \mas r \to \infty.
	}
	To prove this, we argue as in Claim \ref{ell1 claim}.  We insert the decay rate \eqref{v1 to 0} along with the growth rate \eqref{v0 grow} into the difference estimate \eqref{v0 diff1}. We see that for fixed $r_0 \geq R_1$ and all $n \in \N$ we have 
	\begin{align}
	\begin{split}\label{v0 conv}
	\abs{ v_0(2^{n+1}r_0) - v_0(2^n r_0) } &\lesssim 
	(2^n r_0)^{-4} (2^n r_0)^{1/6} + 
	(2^n r_0)^{-4} (2^n r_0)^{1/3} \\&\:\:+ (2^n r_0)^{-4} (2^n r_0)^{1/2} \\
	&\:\:+ (2^n r_0)^{-3} (2^n r_0)^{-2} + 
	(2^n r_0)^{-2} (2^n r_0)^{-4} \\&\:\:+ (2^n r_0)^{-1} (2^n r_0)^{-6} \\
	&\lesssim (2^n r_0)^{-7/2}.
	\end{split}
	\end{align}
	Thus 
	\ant{
		\sum_n \abs{ v_0(2^{n+1}r_0) - v_0(2^n r_0) } < \infty,
	}
	which in turn implies there exists $\ell_0 \in \R$ so that 
	\ant{
		\lim_{n \to \infty} v_0(2^nr_0) = \ell_0.
	}
	In particular, this implies that the sequence $(v_0(2^n r_0))_n$ is bounded.
	Let now $r \geq r_0$ with $2^n r_0 \leq r < 2^{n+1}r_0$.
	The boundedness of the sequence $(v_0(2^n r_0))$ and the difference estimates \eqref{v0 diff1} imply the improvement of \eqref{v0 conv}
	\begin{align*}
	|v_0(2^{m+1}r_0) - v_0(2^m r_0)| \lesssim (2^m r_0)^{-4}
	\end{align*}
	uniformly in $m$ as well as 
	\begin{align*}
	|v_0(r) - v_0(2^n r_0)| \lesssim (2^n r_0)^{-4}.
	\end{align*}
	Combining these inequalities, we conclude
	\begin{align*}
	|\ell_0 - v_0(r)| &\leq \sum_{m = n} |v_0(2^{m+1} r_0) - v_0(2^m r_0)| + 
	|v_0(2^n r_0) - v_0(r)| \\
	&\lesssim (2^n r_0)^{-4} \\
	&\lesssim r^{-4}
	\end{align*}
	which proves the convergence $v_0(r) \rar \ell_0$ as $r \rar \infty$ along with the desired rate of convergence.
\end{proof}

\subsection{Step $3$} We complete the proof of Proposition~\ref{rigid} by showing that $\vec u(t, r)  \equiv (0, 0)$. We separate the argument into two cases depending on whether the number $\ell_0$ in Lemma~\ref{spacial decay} is zero or nonzero.

\begin{flushleft} \textbf{Case 1: $\ell_0=0$ implies $\vec u(t)\equiv(0, 0)$:} 
\end{flushleft}

We formulate this case as a lemma: 
\begin{lem} \label{ell=0 lem}Let $\vec u(t)$ be as in Proposition~\ref{rigid} and let $\ell_0\in \R$ be as in Lemma~\ref{spacial decay}.  If $\ell_0 = 0$ then $\vec u(t)  \equiv (0, 0)$. 
\end{lem}
To prove the lemma, we will first establish that if $\ell_0 = 0$ then $(u_0, u_1)$ must be compactly supported. We then use a ``channels of energy argument" to show that the only compactly supported solution with pre-compact trajectory is $\vec u(t) = (0, 0)$. 
\begin{claim}\label{comp supp}Let $\ell_0$ be as in Lemma~\ref{spacial decay}. If $\ell_0=0$ then $(u_0, u_1)$ is compactly supported. 
\end{claim}
\begin{proof}
	The assumption $\ell_0 = 0$ implies that for $r \ge R_1$, 
	\EQ{ \label{ell0 means}
		\abs{v_0(r)}  = O (r^{-4}) \mas r \to \infty,\\
		\abs{v_1(r)}  = O( r^{-2}) \mas r \to \infty.
	}
	In particular, for $r_0 \ge R_1$ we have the upper bound 
	\EQ{\label{low bound}
		\abs{v_0(2^nr_0)} + \abs{v_1(2^nr_0)} \lesssim (2^nr_0)^{-4} +(2^nr_0)^{-2}.
	}
	We now establish a lower bound.  The difference estimates~\eqref{v0 diff1} and~\eqref{v1 diff1} together with~\eqref{ell0 means} give
	\ant{
		&\abs{v_0(2^{n+1} r_0)} \ge (1-C_1(2^nr_0)^{-4}) \abs{v_0(2^nr_0)} - C_1(2^nr_0)^{-3} \abs{v_1(2^nr_0)}\\
		&\abs{v_1(2^{n+1} r_0)} \ge (1-C_1(2^nr_0)^{-4}) \abs{v_1(2^nr_0)} - C_1(2^nr_0)^{-5} \abs{v_0(2^nr_0)}
	}
	For large $r_0$ we can combine the above estimates to obtain the lower bound 
	\ant{
		\abs{v_0(2^{n+1} r_0)} + \abs{v_1(2^{n+1} r_0)} \ge (1- 2C_1r_0^{-3})( \abs{v_0(2^{n} r_0)} + \abs{v_1(2^{n} r_0)})
	}
	Fixing $r_0$ large enough so that $2C_1r_0^{-3} < \frac{1}{4}$ and arguing inductively we conclude that 
	\ant{
		( \abs{v_0(2^{n} r_0)} + \abs{v_1(2^{n} r_0)}) \ge \left(\frac{3}{4}\right)^{n} ( \abs{v_0(r_0)} + \abs{v_1( r_0)})
	}
	Now, we use~\eqref{low bound} to estimate the left-hand-side above to get 
	\ant{
		\left(\frac{3}{4}\right)^{n} ( \abs{v_0(r_0)} + \abs{v_1( r_0)}) \lesssim 2^{-2n} r_0^{-2}
	}
	which, in turn means that 
	\ant{
		3^n ( \abs{v_0(r_0)} + \abs{v_1( r_0)}) \lesssim 1, 
	}
	which is impossible unless $(v_0(r_0), v_1(r_0)) = (0, 0)$.
	Hence, $$(v_0(r_0), v_1(r_0)) = (0, 0).$$ The relation \eqref{v u project} implies that 
	\ant{
		\| \pi_{r_0} \vec u(0)\|_{ \dot{H}^1 \times L^2(r \ge r_0)} = 0.
	}
	By Lemma~\ref{project ineq} we can also then deduce that 
	\ant{
		\| \pi_{r_0}^{\perp} \vec u(0)\|_{ \dot{H}^1 \times L^2(r \ge r_0)} = 0,
	}
	and hence 
	\ant{
		\|  \vec u(0)\|_{ \dot{H}^1 \times L^2(r \ge r_0)} = 0,
	}
	which concludes the proof since $\lim_{r \to \infty} u_0(r) = 0$. 
\end{proof}

\begin{proof}[Proof of Lemma~\ref{ell=0 lem}] Assume that $\ell_0 = 0$. Then, by Claim~\ref{comp supp}, $(u_0, u_1)$ is compactly supported. Assume, towards a contradiction, that $(u_0, u_1) \not \equiv (0, 0)$. We define $\rho_0>0$ by 
	\ant{
		\rho_0:= \inf\{ \rho \, \, : \, \,  \| \vec u(0)\|_{\dot{H}^1 \times L^2(r \ge \rho)} = 0\}.
	}
	Let $\e>0$ be a small number to be determined below. Choose $\rho_1 = \rho_1( \e)$ so that $\frac{ 1}{2} \rho_0 < \rho_1 < \rho_0$  and 
	\begin{align}\label{rho1 def}
	0< \| \vec u(0)\|_{\dot{H}^1 \times L^2(r \ge \rho_1)}^2  \le \e \le \de_1^2
	\end{align}
	where $\de_1$ is as in~\eqref{R1 de1}. With $(v_0, v_1)$ as in~\eqref{v def} we have 
	\EQ{\label{all small}
		3 \rho_1^{-3} v_0^2(\rho_1) + \rho_1^{-1}v_1^2(\rho_1)&+ \int_{\rho_1}^\I \left( \frac{1}{r} \p_r v_0( r) \right)^2 \, dr + \int_{\rho_1}^\I \left( \p_r v_1(r) \right)^2 \, dr  = \\
		& = \| \pi_{\rho_1} \vec u(0)\|^2_{\dot{H}^1 \times L^2(r \ge \rho_1)}+\| \pi_{\rho_1}^{\perp} \vec u(0)\|^2_{\dot{H}^1 \times L^2(r \ge \rho_1)} \\
		&= \|  \vec u(0)\|^2_{\dot{H}^1 \times L^2(r \ge \rho_1)} < \e
	}

	A simple reworking of the proofs of Lemma \ref{h small data} and Lemma \ref{project ineq} shows that we have the following analog 
	of Lemma \ref{v project}: 
	\begin{align}
	\begin{split}\label{v rho1}
	\Bigl( \int_{\rho_1}^\I \Bigl[\Bigl( \frac{1}{r} \p_r v_0(r) \Bigr)^2 + \Bigl( \p_r v_1( r) \Bigr)^2 \Bigr]\, dr\Bigr)^{\frac{1}{2}}  &\lesssim 
	(\rho_0 - \rho_1)^{2/7} |v_0(\rho_1) | + |v_0(\rho_1)|^2 \\
	&\:\:+ |v_0(\rho_1)|^3 \\
	&\:\:+ (\rho_0 - \rho_1)^{2/7} |v_1(\rho_1) | + |v_1(\rho_1)|^2 \\
	&\:\:+ |v_1(\rho_1)|^3
	\end{split}
	\end{align}
	where the implies constant is independent of $\rho_1$.
	In the original proofs of Lemma \ref{h small data} and Lemma
	\ref{project ineq}, smallness is obtained by taking $R$ sufficiently large and compactness.  To obtain \eqref{v rho1}, smallness is achieved by taking $\e$ and 
	$|\rho_0 - \rho_1|$ sufficiently small, cutting off the potential term to the exterior region $\{ \rho_1 + |t| \leq r \leq \rho_0 + |t| \}$
	and using the compact support of $\vec u(0)$ along with finite speed of propagation.  Since $v_0( \rho_0) = v_1( \rho_0) = 0$ we can argue as in Corollary~\ref{diff esti} and Corollary~\ref{diff esti 2} to obtain
	\ant{
		&\abs{v_0( \rho_1)} = \abs{v_0(\rho_1) - v_0( \rho_0)} \le C_1\, \e\, (\abs{v_0( \rho_1)} + \abs{v_1( \rho_1)})\\
		&\abs{v_1( \rho_1)} = \abs{v_1(\rho_1) - v_1( \rho_0)} \le C_2\, \e\, (\abs{v_0( \rho_1)} + \abs{v_1( \rho_1)})
	}
	where here we have used that $\frac{1}{2} \rho_0< \rho_1 < \rho_0$ to obtain constants $C_1, C_2$ which depend only $\rho_0$ (which is fixed) and the uniform constant in~\eqref{v rho1}, but not on~$\e$. 
	Putting the above estimates together yields
	\EQ{
		(\abs{v_0( \rho_1)} + \abs{v_1( \rho_1)}) \le C_3 \e (\abs{v_0( \rho_1)} + \abs{v_1( \rho_1)})
	}
	which implies that $\abs{v_0( \rho_1)} = \abs{v_1( \rho_1)} = 0$ by taking $\e>0$ small enough. By~\eqref{v rho1} and the equalities in~\eqref{all small} we deduce that 
	\ant{
		\| \vec u(0)\|_{\dot{H}^1 \times L^2(r \ge \rho_1)} = 0, 
	}
	which contradicts \eqref{rho1 def}.  Thus, we must have $\vec u = (0,0)$.
\end{proof}
This completes the proof in the case that $\ell_0 = 0$. We will now show that the case $\ell_0 \neq 0$ is in fact impossible.  
\vspace{\baselineskip}

\begin{flushleft} \textbf{Case 2: $\ell_0 \neq 0$ is impossible:} \end{flushleft}

\begin{lem}\label{u1 imp}
	Let $\ell_0$ be as in Lemma \ref{spacial decay}.  Then necessarily $\ell_0 = 0$.  
\end{lem}

\begin{proof}
	
	We will prove Lemma \ref{u1 imp} by showing that if $\ell_0 \neq0$ then the solution to the Adkins--Nappi wave map equation~\eqref{eq:an} given by $\psi(t, r)  = Q(r) + r u(t, r)$ is equal to a finite energy stationary solution $\tilde Q(r)$ of~\eqref{e00}. However, we know by the uniqueness part of Theorem~\ref{t01} that if $\tilde Q$ is a finite energy solution to \eqref{e00}, then $\tilde Q = Q$ which implies that $\vec u \equiv 0$, a contradiction to our initial assumption that $\ell_0 \neq 0$. 
	
	To prove that our compact solution $\psi$ as above is equal to a solution to the stationary equation \eqref{e00},  we first linearize about the solution given by Lemma \ref{l05a} with the same spatial decay as $\psi$.  We then use the previous arguments that showed that if $\ell_0 = 0$ then $\vec u = (0,0)$ to conclude.   The setup is as follows.  Assume that $\ell_0 \neq 0$.  By Theorem \ref{t01} there exists 
	a unique $\al_n > 0$ such that
	\begin{align*}
	Q(r) = n \pi - \al_n r^{-2} + O(r^{-6}) \quad \mbox{as } r \rar \infty. 
	\end{align*}
	By Lemma \ref{l05a} there exists a unique solution $\vp_{\al_n - \ell_0}$ to the 
	stationary equation 
	\begin{align}\label{Qt equ}
	\vphi_{rr} - \frac{2}{r} \vphi_r = \frac{\sin 2 \vphi}{r^2} + 
	\frac{(\vphi - \sin \vphi \cos \vphi)(1 - \cos 2 \vphi)}{r^4}, \quad r > 0,
	\end{align}
	such that 
	\begin{align}\label{Qell asym}
	\vp_{\al_n - \ell_0}(r) = n\pi - (\al_n - \ell_0)r^{-2} + O(r^{-6}) 
	\quad \mbox{as } r \rar \infty.
	\end{align}
	For brevity we use the notation $\tilde Q = \vp_{\al_n - \ell_0}$.  For each $t \in \R$, we define 
	\begin{align*}
	\tilde u(t,r) = \frac{1}{r} \bigl ( \psi(t,r) - \tilde Q(r) \bigr ), 
	\end{align*}
	where as before, $\vec \psi = \vec Q + r \vec u$ is our compact solution.  We now make some observations about $\tilde u$.  Since $\psi$ satisfies \eqref{eq:an} and 
	$\tilde Q$ satisfies \eqref{Qt equ} we see that 
	$\tilde u$ satisfies 
	\begin{align}
	\begin{split}\label{ul eq}
	\tilde u_{tt} - \tilde u_{rr} - \frac{4}{r} \tilde u + \tilde V(r) \tilde u
	+ \tilde Z(r,r\tilde u) = 0, 
	\end{split}
	\end{align}
	where the potential $\tilde V(r)$ and 
	$\tilde Z(r,\tilde u) = \tilde u^2 \tilde Z_2(r) + \tilde u^3 \tilde Z_3(r,r\tilde u) + \tilde u^4 \tilde Z_4(r,r\tilde u) + \tilde u^5 \tilde Z_5(r\tilde u)$
	are as in \eqref{e32}, \eqref{e33}, \eqref{e34}, \eqref{e35} and \eqref{e36} with $\tilde Q$ in place of $Q$.  In particular, 
	by \eqref{Qell asym} we have the analogs of the estimates 
	\eqref{pot bound} and \eqref{Z bounds}:
	\begin{align}\label{pot bound2}
	|\tilde V(r)| \lesssim r^{-6},
	\end{align}
	as well as 
	\begin{align}
	\begin{split}\label{Z bounds2}
	\bigl | \tilde u^2 \tilde Z_2(r) \bigr | &\lesssim r^{-3} |\tilde u|^2, \\
	\bigl | \tilde u^3 \tilde Z_3(r,r\tilde u) \bigr | &\lesssim |\tilde u|^3, \\
	\bigl | \tilde u^4 \tilde Z_4(r,r\tilde u) \bigr | &\lesssim r^{-1} |\tilde u|^4, \\
	\bigl | \tilde u^5 \tilde Z_5(r\tilde u) \bigr | &\lesssim |\tilde u|^5.
	\end{split}
	\end{align}
	Also, since $\tilde u(t,r) = \frac{1}{r} (\psi(t,r) - \tilde Q(r) ) = 
	u(t,r) + \frac{1}{r} ( Q(r) - \tilde Q(r) )$, we see that $\vec{\tilde u}$ inherits the vanishing property from $\vec u$:
	\begin{align}
	\begin{split}\label{compact ext en2}
	\forall R > 0, \quad \lim_{|t|\rar \infty} \| \vec{\tilde u}(t) \|_{\dot H^1 \times L^2(r \geq 
		R + |t|)} &= 0, \\
	\lim_{R\rar \infty} \sup_{t \in \R} \| \vec{\tilde u}(t) \|_{\dot H^1 \times L^2(r \geq 
		R + |t|)} &= 0.
	\end{split}
	\end{align}
	Finally, by construction we see that
	\begin{align}
	\begin{split}\label{ul limits}
	\tilde v_0(r) &:= r^3 \tilde u(0,r) = O(r^{-4}) \quad \mbox{as } r \rar \infty, \\
	\tilde v_1(r) &:= r \int_r^\infty \tilde u_t(0,\rho) \rho d\rho = O(r^{-2})
	\quad \mbox{as } r \rar \infty.
	\end{split}
	\end{align} 
	
	Using \eqref{ul eq}, \eqref{pot bound2}, \eqref{Z bounds2}, \eqref{compact ext en2} and \eqref{ul limits} we can repeat the previous arguments of this section word for word with $\tilde u$ in place of $u$ to conclude that 
	\begin{align}
	\vec{\tilde u} \equiv (0,0). 
	\end{align}
	Thus, $\psi(t,r) = \tilde Q(r)$, i.e. $\vec \psi$ is a finite energy stationary solution to \eqref{e00}.  By Theorem \ref{t01}, we must have $\vec Q = \vec \psi = 
	\vec Q + r \vec u$ whence $\vec u = (0,0)$, a contradiction to the fact that $\ell_0 \neq 0$.  Thus, $\ell_0$ must be necessarily 0.  
\end{proof}

\subsection{Proof of Proposition~\ref{rigid} and Proof of Theorem~\ref{t:main}}
For clarity, we summarize what we have done in the proof of Proposition~\ref{rigid}. 
\begin{proof}[Proof of Proposition~\ref{rigid}]
	Let $\vec u(t)$ be a solution to~\eqref{e31} and suppose that the trajectory 
	\ant{
		K=\{ \vec u(t) \mid t \in \R\}
	}
	is pre-compact in $ \HH = \dot H^1 \times L^2(\R^5)$.
	By Lemma~\ref{spacial decay} there exists $\ell_0 \in \R$ so that 
	\begin{align} 
	&\abs{r^3 u_0(r) - \ell_0 } = O(r^{-4}) \mas r \to \infty,\\
	&\abs{r \int_r^{\infty} u_1( \rho) \rho \, d\rho} = O(r^{-2}) \mas r \to \infty. 
	\end{align}
	By Lemma \ref{u1 imp}, $\ell_0 = 0$.  Then by Lemma~\ref{ell=0 lem} we can conclude that $\vec u(0)= (0, 0)$, which proves Proposition~\ref{rigid}. 
\end{proof}

The proof of Theorem~\ref{t31} is now complete. We conclude by summarizing the argument. 
\begin{proof}[Proof of Theorem~\ref{t31} (equivalently Theorem~\ref{t:main}(b))]
	Suppose that Theorem \ref{t31} \newline fails. Then by Proposition~\ref{p:cc} there exists a critical element, that is, a nonzero solution $\vec u_{\infty}(t) \in \HH$ to~\eqref{e31} such that the trajectory $K= \{ \vec u_{\infty}(t) \mid t \in \R\}$ is pre-compact in $\HH$. However, Proposition~\ref{rigid} implies that any such solution is necessarily identically equal to $(0, 0)$, which contradicts the fact that the critical element $\vec u_{\infty}(t)$ is, by construction, nonzero. 
\end{proof}

%


\appendix
\section{Strichartz estimates and local energy decay} \label{a:S}

Here we give a brief outline of the proof of Theorem~\ref{t:strich}(b).  The argument is now standard and is based on ideas from~\cite{RodS}, and follows the general philosophy that spectral properties of $H_V$ are deeply related to local energy decay estimates, which then imply Strichartz estimates (as long as we stay away from the $L^2_t$ endpoints).

\begin{proof}[Sketch of the proof of Theorem~\ref{t:strich}\emph{(b)}] 
	First note that it suffices to consider $F=0$ by Minkowski's inequality. 
	Note that   $A:=(-\Delta)^{\frac12}$ satisfies
	\EQ{\label{ta}
		\| A\, f\|_{2} \simeq \|f\|_{\dot H^{1}}  
	}
	for all compactly supported $f\in C^{\infty}(\R^{5})$. 
	For any real-valued $w=(w_{1},w_{2})\in \dot H^{1}\times L^{2}(\R^5)$ we set
	\ant{
		W:= A w_{1} + i w_{2}
	}
	Then \eqref{ta} implies that $\|W\|_{2}\simeq \| (w_{1},w_{2})\|_{\dot{H}^1 \times L^2}$. 
	Furthermore, $w$ solves~\eqref{eq:lw} with $F=0$ if and only if 
	\EQ{
		i\p_{t} W &= A W + V w \\
		W(0) &= A f+ig \in L^{2}(\R^{5})
	}
	Thus 
	\ant{
		W(t) = e^{-it A} W(0) -i \int_{0}^{t} e^{-i(t-s)A} Vw(s)\, ds
	}
	Now, let $X$ be any admissible Strichartz norm (as in the left-hand-side of Theorem~\ref{t:strich}(b)). By Strichartz estimates for the free scalar wave equation in $\R^5$, i.e., for $\Box w = 0$ (see for example~\cite{LinS}),   we have 
	\ant{
		\| A^{-1}\Re e^{-it A} W(0)\|_{X}\lesssim \|W(0)\|_{2}
	}
	Next we factor $V(r)=V_{1}(r)V_{2}(r)$ where each of the factors decays like $r^{-3}$. By the Christ-Kiselev lemma, see~\cite{Sogge, CK01}, and our exclusion of the $L^{2}_{t}$ endpoint, it suffices to control
	\EQ{\label{red}
		\left\| A^{-1}\Re \int_{-\infty}^{\infty} e^{-i(t-s) A}\, V_{1}V_{2} \, w(s)\, ds  \right\|_{X} \le \| K\|_{L^{2}_{t,x}\to X} \|V_{2}\, w(s)\|_{L^{2}_{s,x}} 
	}
	where 
	\EQ{ \label{Kdef} 
		(KF)(t) &:= A^{-1}\Re \int_{-\infty}^{\infty} e^{-i(t-s) A} V_{1}F(s)\, ds.
	}
	Next note that 
	\ant{
		\| KF\|_{X}\le \| A^{-1}\Re e^{-itA} \|_{L^2 \to X} \Big\| \int_{-\infty}^{\infty} e^{is A} V_{1}F(s)\, ds\Big\|_{L^2}.
	}
	The first factor on the right-hand side is some constant by the free Strichartz estimates; see~\cite{LinS}. We claim that the second one is $ \lesssim \|F\|_{L^{2}_{t,x}}$. By duality, this claim is equivalent to
	the {\em localized energy bound}
	\ant{
		\| V_{1}\, e^{-itA} \phi\|_{L^{2}_{t,x}}\le C\|\phi\|_{2}.
	} 
	This is elementary to prove for radial $\phi$  using the
	Fourier transform relative  to $\LL_0 = -\p_{rr}+\frac{2}{r^{2}}$ on $L^{2}((0,\infty))$ after conjugation by $r^2$; see~\cite{LS, LOS1} for examples of how to carry out this standard argument. 
	
	For the second factor in~\eqref{red} we claim the bound
	\EQ{\label{len2}
		\|V_{2}\,  w(s)\|_{L^{2}_{s,x}} \le C \|W(0)\|_{2} = C\|(f,g)\|_{\dot H^{1}\times L^{2}}
	}
	holds for any solution of~\eqref{eq:lw} with $F=0$. 
	One way to prove this is to make use of  distorted Fourier transform relative to the self-adjoint operator $H_V$ 
	on its domain $\calD$,  restricted to radial functions. It is here that the spectral properties of $H_V$ proved in Theorem \ref{t:Hspec} are essential. See for example~\cite[Section 5 and Lemma $5.2$]{LS} or~\cite[Lemma $4.3$]{LOS1} for  detailed arguments in more complicated situations than what's needed here. This completes the sketched proof of Theorem~\ref{t:strich}(b).
\end{proof}

\bibliographystyle{plain}
\bibliography{researchbib}

 \bigskip

\centerline{\scshape Andrew Lawrie}
\smallskip
{\footnotesize
 \centerline{Department of Mathematics, MIT}
\centerline{77 Massachusetts Ave., 2-267, Cambridge, MA 02139, U.S.A.}
\centerline{\email{ alawrie@mit.edu}}
} 

\medskip 

\centerline{\scshape Casey Rodriguez}
\smallskip
{\footnotesize
 \centerline{Department of Mathematics, University of Chicago}
\centerline{5734 South University Avenue, Chicago, IL 60637, U.S.A.}
\centerline{\email{ c-rod216@math.uchicago.edu}}
}

\end{document}